\documentclass[reqno]{amsart}
\usepackage{amsmath}
\usepackage{amssymb}
\usepackage{amsthm}
\usepackage{bbm}
\usepackage{graphicx}
\usepackage{float}
\usepackage[usenames]{color}
\usepackage{cite}
\usepackage{enumerate}
\allowdisplaybreaks

\usepackage{color}

\usepackage[colorlinks,linkcolor=blue,anchorcolor=red,citecolor=blue]{hyperref}
\usepackage{marginnote}

\newcommand{\R}{\mathbb{R}}
\renewcommand{\d}{\mathrm{d}}

\newcommand{\hE}{\hat{E}}
\newcommand{\hmu}{\hat{\mu}}
\newcommand{\hg}{\hat{g}}
\newcommand{\hrho}{\hat{\rho}}
\newcommand{\hf}{\hat{f}}
\newcommand{\hS}{\hat{S}}
\newcommand{\tK}{\tilde{K}}
\newcommand{\hK}{\hat{K}}
\newcommand{\jbr}[1]{\left< #1 \right>}
\newcommand{\G}[1]{G\left[ #1 \right]}
\newcommand{\hnmu}{\widehat{\nabla_v \mu}}
\newcommand{\hAg}{\widehat{Ag}}
\newcommand{\abs}[1]{\left| #1 \right|}
\newcommand{\norm}[1]{\left\| #1 \right\|}
\newcommand{\Aetag}[1]{\left[A\partial_\eta^j\hg_{ #1 }\right]}

\newcommand{\tG}[1]{\widetilde{G}\left[ #1 \right]}
\newcommand{\tF}[1]{\widetilde{F}[ #1 ]}

\numberwithin{equation}{section}

\newtheorem{thm}{Theorem}[section]
\newtheorem{lemma}[thm]{Lemma}
\newtheorem{prop}[thm]{Proposition}
\newtheorem{remark}[thm]{Remark}
\newtheorem{coro}[thm]{Corollary}

\newcommand{\T}{\mathbb{T}}
\newcommand{\Z}{\mathbb{Z}}
\newcommand{\C}{\mathbb{C}}

\renewcommand{\L}{\mathcal{L}}

\makeatletter
\@namedef{subjclassname@2020}{%
	\textup{2020} Mathematics Subject Classification}
\makeatother

\title[Two species Landau damping]{A note on Landau damping of two-species Vlasov-Poisson system}

\author[R.-J. Duan]{Renjun Duan}
\address[R.-J. Duan]{Department of Mathematics, The Chinese University of Hong Kong, Shatin, Hong Kong, P.R.~China}
\email{rjduan@cuhk.edu.hk}

\author[Z. Zhang]{Zhiwen Zhang}
\address[Z. Zhang]{Department of Mathematics, The Chinese University of Hong Kong, Shatin, Hong Kong, P.R.~China}
\email{zwzhang@math.cuhk.edu.hk}

\date{\today}

\subjclass[2020]{35Q83}

\keywords{Vlasov-Poisson system, two-species, Landau damping, regularity}

\begin{document}
	\thispagestyle{empty}
	\begin{abstract}
In this note we adopt an approach by Grenier, Nguyen and Rodnianski  in \cite{GNR} for studying the nonlinear Landau damping of the two-species Vlasov-Poisson system in the phase space $\mathbb{T}^d_x \times \mathbb{R}^d_v$ with the dimension $d\geq 1$. The main goal is twofold: one is to extend the one-species case  to the two-species case where the electron mass is finite and the ion mass is sufficiently large, and the other is to modify the $G$-functional such that it involves the norm in $L^{d+1}$ instead of $L^2$ as well as derivatives up to only the first order.  
	\end{abstract}
	
	\maketitle
	
\section{Introduction}
The Vlasov-Poisson system, a fundamental collisionless kinetic model in plasma physics, describes the time evolution of non-negative velocity distribution functions $F_+(t,x,v)$ and $F_-(t,x,v)$ for positively charged ions and negatively charged electrons, respectively. The governing system for $F_\pm(t,x,v)$ reads as
	\begin{equation}
		\label{e1}
		\left\{
			\begin{aligned}
				& \partial_t F_+ + v \cdot \nabla_x F_+ + \frac{e}{m_+} E \cdot \nabla_v F_+ = 0,\\
				& \partial_t F_- + v \cdot \nabla_x F_- - \frac{e}{m_-} E \cdot \nabla_v F_- = 0,\\
				& E = -\nabla_x \phi, \\
				& -\Delta_x \phi = \int_{\R^d} (F_+ - F_-)\,\d v,
			\end{aligned}
			\right.
	\end{equation}
with $(t,x,v)\in (0,\infty)\times \T^d\times \R^d$.
	In the above equations, $e$ denotes the electron charge, while $m_+$ and $m_-$ represent the masses of ions and electrons, respectively. It is noteworthy that the mass of an electron is significantly less than that of an ion, typically $\frac{m_-}{m_+} \approx 0.005$ in a hydrogen plasma, as highlighted in \cite{FG}. Introducing a new parameter $\varepsilon = \frac{m_-}{m_+}$, defining the transformed distribution functions on the torus as $G_\pm(t,x,v) = F_\pm(t,x,\frac{e}{m_-}v)$ and then taking $e=1$ and $m_-=1$ for simplicity, the system (\ref{e1}) can be reformulated as
	\begin{equation}
		\label{e2}
		\left\{
		\begin{aligned}
			& \partial_t G_+ + v \cdot \nabla_x G_+ + \varepsilon E \cdot \nabla_v G_+ = 0,\\
			& \partial_t G_- + v \cdot \nabla_x G_- - E \cdot \nabla_v G_- = 0,\\
			& E = -\nabla_x \phi, \\
			& -\Delta_x \phi = 
			\int_{\R^d} (G_+ - G_- )\,\d v.
		\end{aligned}
		\right.
	\end{equation}
		
The Vlasov-Poisson system serves as a fundamental model for plasma dynamics. In 1961, S. V. Iordanski\u{\i} \cite{Ior} established the well-posedness of the Vlasov-Poisson system in one dimension. In 1978, S. Ukai and T. Okabe  \cite{UO} extended this result to the two-dimensional case. The three-dimensional problem was independently proved by P.-L. Lions and B. Perthame \cite{LP} and K. Pfaffelmoser \cite{Pfa} (see also \cite{BD2,Scha}). A. A. Arsenev \cite{Arsensev} demonstrated the global existence of weak solutions in three dimensions. E. Horst and R. Hunze \cite{HHN} further improved this result by relaxing the assumptions on the initial data; see also \cite{BD,BDG}.
	
	In the reformulated Vlasov-Poisson system given by the equations in (\ref{e2}), the analysis is extended to understand the behavior of the system near a homogeneous equilibrium characterized by \( G_+ = G_- = \mu(v) \), where \( \mu(v) \) is a given spatially homogeneous distribution function such that \( \int_{\R^d}\mu(v)\d v = 1 \). The perturbations \( f_+ \) and \( f_- \) from this equilibrium are introduced as \( G_\pm = \mu + f_\pm \), leading to the perturbed system:
	\begin{equation}
		\label{e3}
		\left\{
		\begin{aligned}
			& \partial_t f_+ + v \cdot \nabla_x f_+ + \varepsilon E \cdot \nabla_v \mu = -\varepsilon E \cdot \nabla_v f_+,\\
			& \partial_t f_- + v \cdot \nabla_x f_- - E \cdot \nabla_v \mu = E \cdot \nabla_v f_-,\\
			& E = -\nabla_x \phi, \\
			& -\Delta_x \phi = \rho_+ - \rho_-,\\
			& \rho_+ = \int_{\R^d} f_+ \, dv, \ \rho_- = \int_{\R^d} f_- \, \d v,
		\end{aligned}
		\right.
	\end{equation}
supplemented with initial data
$$
f_\pm(0,x,v)=f^0_\pm(x,v).
$$
	The system above is considered under the initial condition ensuring neutrality, as given by:
	\begin{equation}
		\label{initial}
		\int_{\T^d} \int_{\R^d} f^0_+ - f^0_- \, \d v\, \d x = 0,
	\end{equation}
	which reflects the physical requirement that the total charge due to perturbations in the plasma remains neutral overall.
	
	The primary focus of this examination is on the phenomenon of Landau damping, a mechanism by which the electric field \( E \) decays exponentially over time without requiring collisions, merely through the kinetic behavior of the plasma particles. The damping is understood both in the linearized context, where terms involving \( \nabla_v f_\pm \) are neglected, and in the full nonlinear setting.
	
	Landau damping was initially identified in the linear regime by L. D. Landau \cite{Landau} and has since been a pivotal concept in plasma physics. The extension to nonlinear settings for analytic initial data by C. Mouhot and C. Villani \cite{MV} marked a significant advancement in the theory. Their results demonstrated that if the initial perturbations are sufficiently small and analytic, the electric field would decay exponentially even in the nonlinear regime.
	
	Subsequent studies, such as those by J. Bedrossian, N. Masmoudi, and C. Mouhot \cite{BMM}, and further simplified proofs in specific dimensions by E. Grenier, T. T. Nguyen, and I. Rodnianski \cite{GNR}, have broadened the understanding of this phenomenon, including its applicability to situations with less restrictive assumptions on the data (e.g., Gevrey class).
	
	For the two-species system, as discussed by L. Baumann and M. Pirner \cite{BP}, the exploration of linear Landau damping provides foundational insights into how the interactions between species, differentiated by mass scale \( \varepsilon \), affect the damping behavior. This paper aims to extend these analyses by considering how these interactions influence the nonlinear stability and long-time dynamics of the plasma, particularly focusing on how the small parameter \( \varepsilon \) modulates these effects.
	
	\subsection{Equilibria}
	The following conditions (\ref{H1}) and (\ref{H2}) specified for the equilibrium distribution function \( \mu(v) \) are crucial for ensuring the stability and decay properties of the system described by the perturbed Vlasov-Poisson equations. These conditions serve to control the behavior of the system both in Fourier space and under the influence of perturbations:
	
\medskip
	\noindent {\bf Condition \ref{H1}: Analyticity and Decay of Fourier Transform}
	
	Condition (\ref{H1}) requires that the Fourier transform \( \hat{\mu} \) of the distribution function \( \mu \) not only be real analytic but also exhibit exponential decay in the Fourier space. Specifically, it states:
	\begin{equation}
		\tag{H1}
		\label{H1}
		\sum_{|j| \leq 2} \left|\partial^j_\eta \hat{\mu}(\eta)\right| \leqslant C_{\mu} e^{-\theta_0|\eta|},
	\end{equation}
	where \( j \) is a multi-index and \( \eta \) denotes the frequency variable. The exponential decay characterized by \( \theta_0 \) ensures that \( \mu \) has good regularity properties in physical space, which are essential for the analytic treatment of the Vlasov-Poisson equations, particularly when considering perturbations around the equilibrium.
	
	The real analyticity and specified decay rate in the Fourier space imply that \( \mu \) is sufficiently smooth and rapidly decreasing, which helps in controlling the nonlinear terms and ensuring that the perturbative analysis remains valid.

\medskip
	\noindent {\bf Condition \ref{H2}: Penrose Stability Condition}
	
	Condition (\ref{H2}), known as the Penrose stability condition, is a fundamental criterion for the linearized stability of the equilibrium in plasma physics:
	\begin{equation}
		\label{H2}
		\tag{H2}
		\inf_{k \in \Z^d \setminus \{0\}; \Re \lambda \geqslant 0} \left|1 + \int_0^\infty e^{-\lambda t} t \hat{\mu}(kt) \, \d t\right| \geqslant \kappa_0 > 0,
	\end{equation}
	where \( \lambda \) is a complex number with \( \Re \lambda \) denoting its real part. This condition essentially ensures that the plasma response function does not have zeros in the upper half-plane, which would indicate instability or growing modes in the system. The infimum being strictly positive \( (\kappa_0 > 0) \) across all non-zero wave numbers \( k \) in \( \Z^d \) and for all \( \lambda \) with non-negative real parts guarantees that the system is stable against small perturbations.
	
	The Gaussian example \( \mu(v) = e^{-\frac{|v|^2}{2}} \) satisfies these conditions, making it a standard choice in theoretical studies of plasma dynamics. The conditions are also valid more generally for positive, radially symmetric functions in three or more dimensions, broadening the range of potential applications in plasma physics and related fields.
	
	\subsection{Notation}
	Let \( k \in \mathbb{Z}^d \) and \( \eta \in \mathbb{R}^d \). We define the Fourier transform of \( f(t,x,v) \) as
	\[
	\hat{f}(t,k,\eta) = \int_{\mathbb{T}^d} \int_{\mathbb{R}^d} f(t,x,v) e^{-i k \cdot x} e^{-i \eta \cdot v} \, \d v \, \d x,
	\]
	which allows us to express \( f(t,x,v) \) in terms of its Fourier transform by
	\[
	f(t,x,v) = \frac{1}{(2\pi)^{2d}} \sum_{k \in \mathbb{Z}^d} \int_{\mathbb{R}^d} \hat{f}(t,k,\eta) e^{i k \cdot x} e^{i \eta \cdot v} \, \d\eta.
	\]
	Similarly, for the density \( \rho(t,x) \), we have
	\[
	\hat{\rho}(t,k) = \int_{\mathbb{T}^d} \rho(t,x) e^{-i k \cdot x} \, \d x,
	\]
	and consequently,
	\[
	\rho(t,x) = \frac{1}{(2\pi)^d} \sum_{k \in \mathbb{Z}^d} \hat{\rho}(t,k) e^{i k \cdot x}.
	\]
	
	We introduce the concept of {\it generator functions} to quantify the Gevrey regularity of the solution. Define \( g_\pm(t,x,v) = f_\pm(t,x+vt,v) \); then the dynamical equations (\ref{e3}) become:
	\begin{equation}
		\label{eq}
		\left\{
		\begin{aligned}
			&\partial_t g_++\varepsilon E(t,x+vt) \nabla_v \mu(v) =-\varepsilon E(t,x+vt)(\nabla_v-t\nabla_x)g_+,\\
			&\partial_t g_--E(t,x+vt) \nabla_v \mu(v)= E(t,x+vt)(\nabla_v-t\nabla_x)g_-,\\
			&E=-\nabla_x \phi, \ -\Delta_x \phi=\rho=\rho_+-\rho_-,\\
			& \rho_\pm=\int_{\R^d}g_\pm(t,x-vt,v)\,\d v.
		\end{aligned}
		\right.
	\end{equation}
	Let \( z \geq 0 \) denote the analyticity radius, \( \gamma \in (0,1] \) the Gevrey index, \( j \in \mathbb{N}^d \) a multi-index, \( \sigma > d+1 \), and \( \alpha < \frac{1}{d+1} \). For \( \rho \) as specified, we define:
	\begin{equation}
		\label{Fdef}
		F[\rho](t,z)=\sup_{k\in \Z^d\setminus\{0\}}e^{z\left<k,kt\right>^\gamma} |\hat{\rho}(t,k)|\left<k,kt\right>^\sigma |k|^{-\alpha},
	\end{equation}
	and for $g$ in (\ref{eq}), we define
	\begin{multline}
		\label{Gdef}
		G[g(t)](z)=\sum_{|j|\leqslant 1}\sum_{k\in\Z^d}\int_{\R^d} e^{(d+1)z\left<k,\eta\right>^\gamma}\left[ \left|\partial^j_\eta \hat{g}_+(k,\eta)\right|^{d+1}\right.\\
		\left.+\left|\partial^j_\eta \hat{g}_-(k,\eta)\right|^{d+1} \right] \left<k,\eta\right>^{(d+1)\sigma} \d \eta.
	\end{multline}
	
	The definition of generator functions \( F[\rho] \) and \( G[g] \) differs from that in \cite{GI}. Though, it remains valid that \( F[\rho](t,z) \leqslant C_0 G[g(t)]^{\frac{1}{d+1}}(z) \) as demonstrated in Lemma \ref{FC0G}. In particular, the $G$-functional in \eqref{Gdef} involves the norm in $L^{d+1}$ instead of $L^2$ as well as derivatives in $\eta$ up to only the first order. 
	
	\subsection{Main result}
	In the conventional methodology for validating nonlinear Landau damping, our initial step involves establishing a parallel result within the context of the linearized two-species Vlasov-Poisson system, as delineated by the following set of equations:
	\begin{equation}
		\label{lineareq}
		\left\{
		\begin{aligned}
			&\partial_t g_+(t,x,v)+\varepsilon E(t,x+vt)\nabla_v\mu(v)=0, \\
			&\partial_t g_-(t,x,v)- E(t,x+vt)\nabla_v\mu(v)=0,\\
			& E=-\nabla_x \phi, \ -\Delta_x \phi=\rho=\rho_+-\rho_-,\\
			& \rho_\pm=\int_{\R^d} g_\pm (t,x-vt,v)\,\d v.
		\end{aligned}
		\right.
	\end{equation}

	\begin{thm}
		\label{linearthm}
		Consider the linearized Vlasov-Poisson system as specified in Equation (\ref{lineareq}). Assuming that the parameter $\varepsilon$ is sufficiently small and that the Penrose conditions (\ref{H1},\ref{H2}) are satisfied, let $\hrho_\pm(t,k)$ represent the solutions to this linear problem, and let $F[\rho](t,z)$ denote the associated generator function. Then, for any $\gamma \in (0,1]$, the following inequality holds:
		\begin{equation}
			\label{estFr}
			F[\rho](t,z)\leqslant F[S] (t,z) +C\int_0^t e^{-\frac{1}{4}\theta_1(t-s)}F[S](s,z)\,\d s
		\end{equation}	
		for any $z \in \left[0,\frac{\theta_1}{2}\right]$, where $\theta_1$ and $C$ are universal constants and $S(t,x)$ is defined through the Fourier transform $\hS(t,k)=\hf^0_+(k,kt)-\hf_-^0(k,kt)$.
	\end{thm}
	
	This result mirrors those presented by L. Baumann and M. Pirner \cite{BP}, which we further generalize within the Gevrey class framework. One essential assumption in Theorem \ref{linearthm} is smallness of $\varepsilon=\frac{m_-}{m_+}$ with $m_-=1$, meaning that the electron mass is finite and the ion mass is sufficiently large. 
	
	Subsequently, we deduce the occurrence of non-linear Landau damping in the two-species Vlasov-Poisson system \eqref{e3}. As seen from the proof later, we note that smallness of $\varepsilon$ is not necessary for the non-linear damping provided that Theorem \ref{linearthm} holds true.	
	
	\begin{thm}
		\label{mainthm}
		Consider the Vlasov-Poisson system delineated in Equation (\ref{e3}). Let $\varepsilon>0$ be the constant satisfying Theorem \ref{linearthm}, and let $\mu$ represent a homogeneous equilibrium that satisfies the hypotheses (H1) and (H2) with $\int_{\R^d}\mu(v)\d v=1$. Assume $\lambda_1>0$ and $\gamma=1$. There exists a sufficiently small $\varepsilon_0$ such that for any initial conditions $f^0_\pm$ fulfilling (\ref{initial}) and
		\begin{equation}
			\label{Ginitial}
			G[f^0](\lambda_1)\leqslant \varepsilon_0,
		\end{equation}
		where $G$ is defined as specified in (\ref{Gdef}) with $\sigma>\max\{d+1,3\}$ and $0<\alpha<\frac{1}{d+1}$, the Landau damping can be observed. Specifically, for the unique solution of (\ref{eq}), it holds that 
		\begin{equation*}
			\G{g(t)}(\lambda(t))\leqslant C\varepsilon_0,
		\end{equation*}
		and
		\begin{equation*}
			F[\rho](t,\lambda(t))\leqslant C\varepsilon_0^{\frac{1}{d+1}},
		\end{equation*}
		for any $t>0$. Here, the generator functions $F,G$ are respectively defined in (\ref{Fdef}) and (\ref{Gdef}). The function $\lambda(t)=\lambda_0+\lambda_0(1+t)^{-\delta}$ is defined for some suitable $\lambda_0>0$ and $\delta\in(0,1)$. Consequently, both the force field $E$ and the density $\rho$ asymptotically tends to zero at an exponential rate as time goes to infinity.
	\end{thm}
	
	\begin{remark}
		In the one-dimensional scenario where $d=1$, the Gevrey class parameter $\gamma$ in Theorem \ref{mainthm} can be extended to $\gamma \in \left(\frac{1}{3}, 1\right]$. Under these conditions, the primary estimate Lemma \ref{L41} remains valid for $\gamma \in \left(\frac{1}{3},1\right]$, as discussed in \cite{GNR}.
	\end{remark}
	
To the end, the constant \( C \) is mutable and may depend on initial conditions \( f_0 \), the spatial dimension \( d \), and the equilibrium state \( \mu \).
	
	\section{Preliminary}
	\subsection{Equilibria}
	In this section, we firstly establish the persistence of the Penrose stability condition under the presence of a perturbatively small parameter $\varepsilon>0$ for the two-species model.
	
\begin{lemma}\label{H2l}
Consider the two-species Vlasov-Poisson system as described by Equation (\ref{eq}) with $0<\varepsilon<\alpha_0$, where $\alpha_0>0$ is a constant such that
 $\theta_0^2 > \frac{2\alpha_0C_{\mu}}{\kappa_0}$.
 Then, for every $\alpha \in [0, \alpha_0]$, the following inequality holds:
		\begin{equation}
			\label{H3}
			\inf_{k \in \mathbb{Z}^d \setminus {0}; \Re(\lambda) \geq 0} \left| 1 + (\alpha + 1) \int_0^\infty e^{-\lambda t} t \hat{\mu}(kt)\,\d t \right| \geq \frac{\kappa_0}{2} > 0.
		\end{equation}
	\end{lemma}
	\begin{proof}
		To validate the lemma, we direct our focus towards the integral term in Equation (\ref{H3}). Given that $\Re(\lambda) \geqslant 0$ and $k \in \mathbb{Z}^d \setminus {0}$, and recalling Equation (\ref{H1}), we have
		\begin{equation*}
			\begin{aligned}
				\left| \int_0^\infty e^{-\lambda t} t \hat{\mu}(kt)\,\d t \right| &\leq C_{\mu} \int_0^\infty |e^{-\lambda t}| t e^{-\theta_0 |k| t}\,\d t \\
				&\leq C_{\mu} \int_0^\infty t e^{-\theta_0 |k| t} \d t \\
				&= \frac{C_{\mu}}{\theta_0^2 |k|^2} \leqslant \frac{C_{\mu}}{\theta_0^2},
			\end{aligned}
		\end{equation*}
		where we have utilized the decay properties of the exponential function and the integrability of $t e^{-\theta_0 |k| t}$ over $[0, \infty)$. Subsequently, for $\alpha \in [0, \alpha_0]$, we deduce that
		\begin{equation*}
			\begin{aligned}
				& \left| 1 + (\alpha + 1) \int_0^\infty e^{-\lambda t} t \hat{\mu}(kt)\,\d t \right| \\
				 \geqslant & \left| 1 + \int_0^\infty e^{-\lambda t} t \hat{\mu}(kt) \,\d t  \right| - \alpha_0 \left| \int_0^\infty e^{-\lambda t} t \hat{\mu}(kt) \,\d t  \right|\\
				\geqslant & \kappa_0 - \frac{\alpha_0 C_{\mu}}{\theta_0^2} \geq \frac{\kappa_0}{2},
			\end{aligned}
		\end{equation*}
		thus confirming the stability condition \eqref{H3} and completing the proof of  Lemma \ref{H2l}. 
	\end{proof}
	
	\subsection{Properties of generator functions}
	We define the family of functions 
	$$
	A_{k,\eta}=e^{z\jbr{k,\eta}^\gamma} \jbr{k,\eta}^\sigma.
	$$
	This allows us to express the generator functions \eqref{Fdef} and \eqref{Gdef} as
	\[F[\rho](t,z)=\sup_{k\in \Z\setminus\{0\}}A_{k,\eta} |\hat{\rho}(t,k)| |k|^{-\alpha},\]
	and
	\[G[g(t)](z)=\sum_{|j|\leqslant 1}\sum_{k\in\Z^d}\int_{\R^d} A^{d+1}_{k,\eta}\left[ \left|\partial^j_\eta \hat{g}_+(k,\eta)\right|^{d+1}+\left|\partial^j_\eta \hat{g}_-(k,\eta)\right|^{d+1} \right] \d \eta,\]
	respectively. To the end we denote $\widehat{Ag}(t,k,\eta)=A_{k,\eta}\hat{g}(t,k,\eta)$.
	
	\begin{lemma}
		\label{rhog}
		For any $t\geqslant 0$ and $k\in \Z^d$,
		\begin{equation}
			\label{rhotk}
			\hat{\rho}_\pm(t,k)=\hat{g}_\pm(t,k,kt).
		\end{equation}
	\end{lemma}
	\begin{proof}
		Utilizing the definition of the Fourier transform, we find
		\begin{equation*}
			\begin{aligned}
				\hat{\rho}_\pm(t,k) &= \int_{\T^d} \rho_\pm (t,x)e^{-ikx} \d x\\
				&=\int_{\T^d}\int_{\R^d}g_\pm(t,x-vt,v)e^{-ikx}\d v\d x \\
				&=\int_{\T^d}\int_{\R^d}g_\pm(t,x-vt,v)e^{-ik(x-vt)-ivkt}\d v\d x\\
				&=\hat{g}_\pm (t,k,kt),
			\end{aligned}
		\end{equation*}
		which gives \eqref{rhotk}.
	\end{proof}
	
	\begin{lemma}
		\label{FC0G}
		Let $\lambda_1$ be as defined in (\ref{Ginitial}). There exists a constant $C_0$, depending on $\lambda_1$, such that for any $z \in [0, \lambda_1]$,
		\begin{equation}
			\label{FG}
			F[\rho](z) \leqslant C_0 \G{g(t)}^{\frac{1}{d+1}}(z).
		\end{equation}
	\end{lemma}
	
	\begin{proof}
		We initiate our proof by recalling the bounds established by \cite[eq. (2.10)]{BMM} and \cite[eq. (2.4)]{GI}, which state that for each $z \in [0, \lambda_1]$,
		\begin{equation}
			\label{i1}
			\left|\partial_\eta^j A_{k,\eta}\right| \leqslant \bar{C}(j) \frac{1}{\langle k,\eta \rangle^{|j|(1-\gamma)}} A_{k,\eta} \leqslant \bar{C}(j) A_{k,\eta}.
		\end{equation}
		By Lemma \ref{rhog}, the Fourier transform of $\rho(t)$, denoted $\hat{\rho}(t,k)$, is expressed as $\hat{\rho}(t,k) = \hat{g}_+(t,k,kt) - \hat{g}_-(t,k,kt)$. Consequently,
		\begin{equation}
			\label{i2}
			\begin{aligned}
				A_{k,kt} \left|\hat{\rho}(t,k)\right| \left|k\right|^{-\alpha}  &\leqslant A_{k,kt} \left|\hat{g}_+(t,k,kt) - \hat{g}_-(t,k,kt)\right| \left|k\right|^{-\alpha} \\
				&\leqslant \sup_{\eta} A_{k,\eta} \left|\hat{g}_+(t,k,\eta) - \hat{g}_-(t,k,\eta)\right| \\
				&\leqslant \left\|\widehat{A g}_+(t,k,\eta) - \widehat{A g}_-(t,k,\eta)\right\|_{W^{1,d+1}_\eta}.
			\end{aligned}
		\end{equation}
		Here, in the final inequality we have employed the $L^\infty$ Sobolev embedding.
		It then follows from \eqref{i2} that
		\begin{equation}
			\label{i3}
			A_{k,kt} \left|\hat{\rho}(t,k)\right| \left|k\right|^{-\alpha}
			\leqslant \sum_{|j| \leqslant 1} \left(\norm{\partial_\eta^j \widehat{A g}_+(t,k,\eta)}_{L^{d+1}_\eta} + \norm{\partial_\eta^j \widehat{A g}_-(t,k,\eta)}_{L^{d+1}_\eta}\right).
		\end{equation}
		Note 
		$$
		\sum_{|j|\leqslant 1}\norm{\partial_\eta^j\hAg_\pm(t,k,\eta)}_{L^{d+1}_\eta}
				\leqslant C\left(\sum_{|j|\leqslant 1}\norm{\partial_\eta^j\hAg_\pm(t,k,\eta)}_{L^{d+1}_\eta}^{d+1}\right)^{\frac{1}{d+1}},
		$$
		and 
		\begin{equation*}
			\begin{aligned}
			&\sum_{|j|\leqslant 1}\norm{\partial_\eta^j\hAg_\pm(t,k,\eta)}_{L^{d+1}_\eta}^{d+1}\\
				&\leqslant \int_{\R^d}\abs{\hAg_\pm(t,k,\eta)}^{d+1}\d \eta +\sum_{|j|= 1}\int_{\R^d}\abs{\left(\partial_\eta^j A_{k,\eta}\right)\hg_\pm(t,k,\eta)}^{d+1}\d \eta \\
				&\qquad+ \sum_{|j|= 1}\int_{\R^d}\abs{A_{k,\eta}\partial_\eta^j\hg_\pm(t,k,\eta)}^{d+1}\d \eta.
			\end{aligned}
		\end{equation*}
		Hence, utilizing (\ref{i1}), the above estimates further imply that		
		\begin{equation}
			\label{i4}
			\begin{aligned}
				& \sum_{|j|\leqslant 1}\norm{\partial_\eta^j\hAg_\pm(t,k,\eta)}_{L^{d+1}_\eta}\\
				\leqslant & C\left(\int_{\R^d}\abs{\hAg_\pm(t,k,\eta)}^{d+1}\d \eta+\sum_{|j|= 1}\int_{\R^d}\abs{A_{k,\eta}\partial_\eta^j\hg_\pm(t,k,\eta)}^{d+1}\d \eta\right)^{\frac{1}{d+1}}\\
				\leqslant & C\G{g(t)}^{\frac{1}{d+1}}(z).
			\end{aligned}
		\end{equation}
		By combining (\ref{i2}), (\ref{i3}), and (\ref{i4}), we affirm the validity of inequality (\ref{FG}).
		\end{proof}
	
		\begin{coro}
		\label{re23}
		By the last inequality of (\ref{i2}) and (\ref{i3}, \ref{i4}), we have for any $k\in\Z^d, t\geqslant 0$,
		\[\sup_{\eta}A_{k,\eta}\left|\hat{g}_+(t,k,\eta)-\hat{g}_-(t,k,\eta)\right|\leqslant CG[g(t)]^{\frac{1}{d+1}}(z).\] 
		\end{coro}
	
	\begin{coro}
		\label{re24}
		In Corollary \ref{re23}, let $t=0$, then note that $g(0,x,v)=f(0,x,v)$, and thus it holds that for any $k\in \Z^d$,
		\begin{equation*}
			\label{g0kkt}
			\begin{aligned}
				A_{k,kt}\left|\hat{g}_+(0,k,kt)-\hat{g}_-(0,k,kt)\right| &\leqslant \sup_{\eta}A_{k,\eta}\left|\hat{g}_+(0,k,\eta)-\hat{g}_-(0,k,\eta)\right|\\
				&\leqslant CG[f^0]^{\frac{1}{d+1}} (z).
			\end{aligned}
		\end{equation*}
		Moreover, since $F[\rho](t,z)$ is monotone increasing in $z$, assuming that $\mu$ and $f^0$ satisfy the assumptions in Theorem \ref{mainthm}, we have for $z\in [0,\lambda_1]$,
		\begin{equation*}
			\label{Frho0}
			F[\rho](0,z)\leqslant F[\rho](0,\lambda_1)\leqslant CG[f^0]^{\frac{1}{d+1}}(z)\leqslant C\varepsilon_0^\frac{1}{d+1}.
		\end{equation*}
	\end{coro}
	
	\section{Linear Landau damping}
	In this section, we prove the linear landau damping of the system (\ref{eq}) based on the generator functions. The study via resolvent estimates is classical, see \cite{GNR,Degond,GS,HNR,LZ}.
	
	\subsection{Equation on the density}
	We introduce the linear Vlasov-Poisson system around the equilibrium $\mu(v)$ as described in equation (\ref{lineareq}):
	\begin{equation}
		\label{lineareqs3}
		\left\{
		\begin{aligned}
			&\partial_t g_+(t,x,v)+\varepsilon E(t,x+vt)\nabla_v\mu(v)=0, \\
			&\partial_t g_-(t,x,v)- E(t,x+vt)\nabla_v\mu(v)=0,\\
			& \nabla\cdot E=\rho_+-\rho_-=\rho,\\
			& \rho_\pm=\int_{\R^d} g_\pm (t,x-vt,v)\d v.
		\end{aligned}
		\right.
	\end{equation}
	
	To solve (\ref{lineareqs3}), we follow the standard strategy and first derive a closed equation on the electric field. Let $\hat{\rho}_\pm(t,k)$ be the Fourier transform of $\rho_\pm(t,x)$ in $x$, and $\hat{g}_\pm(t,k,\eta)$ be the Fourier transform of $g(t,x,v)$ in $x$ and $v$. Note that as $\hat{\rho}_\pm(t,0)=0$ for all times, throughout this article, we shall only focus on the case when $k\neq 0$. We have the following lemma.
	
	\begin{lemma}
		Let $g_\pm$ be the unique solution to the linear problem (\ref{lineareqs3}). There holds the following closed equation on the density
		\begin{equation}
			\label{lde1}
			\hat{\rho}_+(t,k)+\varepsilon\hat{\rho}_+(t,k)*_t t\hat{\mu}(kt)-\varepsilon \hat{\rho}_-(t,k)*_t t\hat{\mu}(kt)=\hat{S}_+(t,k),
		\end{equation}
		\begin{equation}
			\label{lde2}
			\hat{\rho}_-(t,k)-\hat{\rho}_+(t,k)*_t t\hat{\mu}(kt)+ \hat{\rho}_-(t,k)*_t t\hat{\mu}(kt)=\hat{S}_-(t,k),
		\end{equation}
		with the source term
		\[\hat{S}_+(t,k)=\hat{f}^0_+(k,kt) \ \ \hat{S}_-(t,k)=\hat{f}^0_-(k,kt).\]
	\end{lemma}
	\begin{proof}
		Take Fourier transform of $E(t,x+vt)\nabla_v\mu(v)$ to get
		\begin{equation}
			\label{emu}
			\begin{aligned}
				& \int_{\T^d}\int_{\R^d}e^{-ikx-i\eta v}E(t,x+vt)\nabla_v\mu(v)\d v\d x\\
				= & \int_{\T^d}\int_{\R^d} e^{-ik(x+vt)-iv(\eta-kt)}E(t,x+vt)\nabla_v\mu(v)\d (x+vt)\d v\\
				= & \hat{E}(t,k)\widehat{\nabla_v \mu}(\eta-kt)\\
				=& i(\eta-kt) \hat{E}(t,k)\hat{\mu}(\eta-kt).
			\end{aligned}
		\end{equation}
		Using (\ref{emu}), the Fourier transform of the first and second equations in (\ref{lineareqs3}) gives
		\begin{equation}
			\label{gemu1}
			\partial_t \hat{g}_+(t,k,\eta)+\varepsilon i(\eta-kt)\hat{E}(t,k)\hat{\mu}(\eta-kt)=0,
		\end{equation}
		\begin{equation}
			\label{gemu2}
			\partial_t \hat{g}_-(t,k,\eta)- i(\eta-kt)\hat{E}(t,k)\hat{\mu}(\eta-kt)=0.
		\end{equation}
		Let's integrate equation (\ref{gemu1}) with respect to $t$ to obtain
		\begin{equation*}
			\hat{g}_+(t,k,\eta)-\hat{g}_+(0,k,\eta)+\varepsilon \int_{0}^{t}i(\eta-ks)\hat{E}(s,k)\hat{\mu}(\eta-ks) \d s=0.
		\end{equation*}
		Let $\eta=kt$. Since $\hg_+(0,k,kt)=\hat{f}_+^0(k,kt)=\hS_+(t,k)$, we have
		\[\hg_+(t,k,kt)+\varepsilon \int_0^t(t-s)ik\hE(s,k)\hmu(kt-ks)\d s=\hS_+(t,k).\]
		Then by (\ref{rhotk}) and $ik\hE(s,k)=\hrho(s,k)$ (by the third equation of (\ref{lineareqs3})), we  deduce (\ref{lde1}). By applying the same integration procedure to equation (\ref{gemu2}), we can deduce equation (\ref{lde2}).
	\end{proof}
	
	\subsection{Resolvent estimates}
	In this section we introduce the Penrose condition in order to solve (\ref{lineareqs3}). For any function $F$ in Lebesgue space $L^2(\R_+)$, we recall that the Laplace transform of $F(t)$ is defined by
	\[\L[F](\lambda)=\int_{0}^{\infty}e^{-\lambda t}F(t)\d t,\]
	which is well-defined for any complex $\lambda$ with $\Re \lambda>0$. Taking the Laplace transform of equations (\ref{lde1}) and (\ref{lde2}) with respect to the variable $t$, we get
	\[\L[\hrho_+](\lambda,k)+\varepsilon \L[\hrho_+](\lambda,k)\L[t\hmu(kt)](\lambda,k)-\varepsilon\L[\hrho_-](\lambda,k)\L[t\hmu(kt)](\lambda,k)=\L[\hS_+](\lambda,k),\]
	\[\L[\hrho_-](\lambda,k)- \L[\hrho_+](\lambda,k)\L[t\hmu(kt)](\lambda,k)+\L[\hrho_-](\lambda,k)\L[t\hmu(kt)](\lambda,k)=\L[\hS_-](\lambda,k),\]
	which means
	\begin{equation}
		\label{lrho1}
		\L[\hrho_+](\lambda,k)=\L[\hS_+](\lambda,k)\left(1-\varepsilon \tK(\lambda,k)\right)+\L[\hS_-]\left(\varepsilon \tK\right),
	\end{equation}
	\begin{equation}
		\label{lrho2}
		\L[\hrho_-](\lambda,k)=\L[\hS_+]\tK(\lambda,k)+\L[\hS_-](\lambda,k)\left(1-\tK(\lambda,k)\right),
	\end{equation}
	where
	\[\tK(\lambda,k)=\frac{\L[t\hmu(kt)](\lambda,k)}{1+(\varepsilon+1)\L[t\hmu(kt)](\lambda,k)}.\]
	Lemma \ref{H2l} ensures that the symbol $1+(\varepsilon+1)\L[t\hmu(kt)](\lambda,k)$ is not degenerate. More precisely, we have that for any $\alpha\in[0,\alpha_0]$,
	\begin{equation}
		\label{con2}
		\inf_{k\in \Z^d\setminus\{0\}; \Re \lambda\geqslant 0} \left|1+(\alpha+1)\L[t\hmu(kt)](\lambda)\right|\geqslant \frac{\kappa_0}{2}>0
	\end{equation}
	for some positive constant $\kappa_0$.
	
	Thus in order to derive pointwise estimates for $\hrho(t,k)$, we first derive bounds on the resolvent kernel $\tK(\lambda,k)$.
	
	\begin{lemma}
		\label{esttK}
		Assume the Penrose conditions (\ref{H1},\ref{H2}) hold. Let $\lambda\in \C$, there is a pointwise constant $\theta_1<\theta_0$, such that $\tK(\lambda,k)$ is an analytic function in $\{\Re \lambda\geqslant -\theta_1|k|\}$. In addition, there is a universal constant $C$ such that
		\begin{equation*}
			\left|\tK(\lambda,k)\right|\leqslant \frac{C}{1+|k|^2+|\Im\lambda|^2}
		\end{equation*}
		uniformly in $\lambda$ and $k\neq 0$ such that $\Re \lambda=-\theta_1 |k|$. The constants $\theta_1$ and $C$ are both exclusively determined by $\mu$.
	\end{lemma}
	
	\begin{proof}
		Use the same argument as in the proof of \cite[Lemma 3.2]{GNR}. Note that the modulus of the denominator of $\tK(\lambda,k)$ has a positive lower bound, by (\ref{con2}).
	\end{proof}
	
	\subsection{Pointwise estimates}
	\begin{lemma}
		Assume that Penrose conditions (\ref{H1},\ref{H2}) hold. Then the unique solution $\hrho_\pm(t,k)$ can be expressed by
		\begin{equation}
			\label{L33r1}
			\hrho_+(t,k)=\hS_+(t,k)-\varepsilon \int_0^t\hK(t-s,k)\hS_+(s,k)\d s+\varepsilon \int_0^t \hK (t-s,k) \hS_-(s,k)\d s,
		\end{equation}
		\begin{equation}
			\label{L33r2}
			\hrho_-(t,k)=\hS_-(t,k)+ \int_0^t\hK(t-s,k)\hS_+(s,k)\d s- \int_0^t \hK (t-s,k) \hS_-(s,k)\d s,
		\end{equation}
		where the kernel $\hK(t,k)$ satisfies
		\begin{equation}
			\label{L33K}
			\left|\hK(t,k)\right|\leqslant C e^{-\theta_1|kt|}
		\end{equation}
		for some constant C.
	\end{lemma}
	\begin{proof}
		Estimates (\ref{L33r1},\ref{L33r2}) can be directly derived from (\ref{lrho1},\ref{lrho2}), with
		\[\tK(\lambda,k)=\L\left[\hK(t,k)\right](\lambda).\]
		To prove (\ref{L33K}), we only need to use the Laplace inverse transform and Lemma \ref{esttK}. A detailed analysis may be referred to the \cite[Proposition 3.3]{GNR}.
	\end{proof}
	
	\subsection{Gevrey estimates}
	Let
	\begin{equation*}
		\label{Sdef}
		\hS(t,k)=\hS_+(t,k)-\hS_-(t,k).
	\end{equation*}
	Then by (\ref{L33r1},\ref{L33r2}), we have
	\begin{equation}
		\label{rdef}
		\begin{aligned}
			\hrho(t,k)&=\hrho_+(t,k)-\hrho_-(t,k)\\
			&=\hS(t,k)-(1+\varepsilon) \int_0^t\hK(t-s,k)\hS(s,k)\d s.
		\end{aligned}
	\end{equation}
	Now we are ready to prove Landau damping of linearized Vlasov-Poisson system.
	
	\begin{proof}[Proof of Theorem \ref{linearthm}]
		Referring to equation (\ref{rdef}) and equation (\ref{L33K}), for any $k\neq 0$, we can observe that
		\begin{equation}
			\label{T34e1}
			\begin{aligned}
				& e^{z\jbr{k,kt}^\gamma}\left|\hrho(t,k)\right|\jbr{k,kt}^\sigma\\
				\leqslant & e^{z\jbr{k,kt}^\gamma} \jbr{k,kt}^\sigma \left[\left|\hS(t,k)\right|+(1+\varepsilon)\int_0^t \abs{\hK(t-s,k)}\abs{\hS(s,k)}\d s\right] \\
				\leqslant & e^{z\jbr{k,kt}^\gamma} \jbr{k,kt}^\sigma \left[\left|\hS(t,k)\right|+C\int_0^t e^{-\theta_1|k|(t-s)}\abs{\hS(s,k)}\d s\right].
			\end{aligned}
		\end{equation}
		It is sufficient to treat the time integral term, since another term is exactly $F[S](t,z)$.
		
		Our goal is to prove that 
		\begin{equation}
			\label{T34e0}
			e^{z\jbr{k,kt}^\gamma} \jbr{k,kt}^\sigma e^{-\theta_1\abs{k}(t-s)}\leqslant C e^{-\frac{1}{4}\theta_1\abs{k}(t-s)} e^{z\jbr{k,ks}^\gamma} \jbr{k,ks}^\sigma .
		\end{equation}
		To prove (\ref{T34e0}), we firstly treat the exponential term. As $z\in\left[0,\frac{1}{2}\theta_1\right]$ and $\gamma\in (0,1]$, we have
		\begin{equation*}
			\begin{aligned}
				z\jbr{k,kt}^\gamma-z\jbr{k,ks}^\gamma &\leqslant \frac{1}{2}\theta_1\left(\jbr{k,kt}-\jbr{k,ks}\right) \\
				&= \frac{1}{2}\theta_1\frac{\left(1+|k|^2+|kt|^2\right)-\left(1+|k|^2+|ks|^2\right)}{\sqrt{1+|k|^2+|kt|^2}+\sqrt{1+|k|^2+|ks|^2}}\\
				&\leqslant \frac{1}{2}\theta_1\frac{|k|^2(t^2-s^2)}{|kt|+|ks|}\\
				&=\frac{1}{2}\theta_1|k|(t-s),
			\end{aligned}
		\end{equation*}
		which means
		\begin{equation}
			\label{T34e2}
			e^{z\jbr{k,kt}^\gamma}e^{-\frac{1}{2}\theta_1|k|(t-s)}\leqslant e^{z\jbr{k,ks}^\gamma}.
		\end{equation}
		
		Secondly, we pay attention to the polynomial term in (\ref{T34e0}). We claim that
		\begin{equation}
			\label{T34e3}
			\jbr{k,kt}^\sigma e^{-\frac{1}{4}\theta_1|k|(t-s)}\leqslant C\jbr{k,ks}^\sigma
		\end{equation}
		for some universal constant $C$ independent of $k$ and $t$. To prove equation (\ref{T34e3}), we take the logarithm of both sides, resulting in the transformation of equation (\ref{T34e3}) to
		\begin{equation}
			\label{T34e4}
			\frac{\sigma}{2}\ln \frac{1+|k|^2+|kt|^2}{1+|k|^2+|ks|^2}-\frac{1}{4}\theta_1|k|(t-s)\leqslant \ln C.
		\end{equation}
		Now we only need to prove (\ref{T34e4}). If $s\leqslant \frac{1}{2}t$, note that $|k|\geqslant 1$
		\begin{equation}
			\label{T34e5}
			\begin{aligned}
				\frac{\sigma}{2}\ln \frac{1+|k|^2+|kt|^2}{1+|k|^2+|ks|^2}-\frac{1}{4}\theta_1|k|(t-s)\leqslant &\frac{\sigma}{2}\ln \frac{1+|k|^2+|kt|^2}{1+|k|^2}-\frac{1}{8}\theta_1|k|t\\
				\leqslant &\frac{\sigma}{2}\ln \frac{|k|^2+|kt|^2}{|k|^2}-\frac{1}{8}\theta_1t\\
				\leqslant & C.
			\end{aligned}
		\end{equation}
		If $s>\frac{1}{2}t$,
		\begin{equation}
			\label{T34e6}
			\begin{aligned}
				\frac{\sigma}{2}\ln \frac{1+|k|^2+|kt|^2}{1+|k|^2+|ks|^2}-\frac{1}{4}\theta_1|k|(t-s) \leqslant &\frac{\sigma}{2}\ln\frac{|kt|^2}{|ks|^2}-\frac{1}{4}\theta_1|k|(t-s)\\
				= &\sigma \left(\ln t-\ln s\right)-\frac{1}{4}\theta_1 (t-s)\\
				=& \left(\sigma\ln t-\frac{1}{4}\theta_1 t\right)-\left(\sigma \ln s-\frac{1}{4}\theta_1 s\right)\\
				\leqslant & 2\sup_{t>0}\left(\sigma\ln t-\frac{1}{4}\theta_1 t\right)\\
				\leqslant & C.
			\end{aligned}
		\end{equation}
		Combining (\ref{T34e5}) and (\ref{T34e6}), we prove (\ref{T34e4}). Then we derive (\ref{T34e3}). (\ref{T34e0}) holds by (\ref{T34e2}) and (\ref{T34e3}).
		
		Finally, substituting equation (\ref{T34e0}) back into equation (\ref{T34e1}), we get
		\begin{equation}
			\label{T34e7}
			\begin{aligned}
				& e^{z\jbr{k,kt}^\gamma}\left|\hrho(t,k)\right|\jbr{k,kt}^\sigma\\
				\leqslant & e^{z\jbr{k,kt}^\gamma} \jbr{k,kt}^\sigma \left|\hS(t,k)\right|+C\int_0^t e^{-\frac{1}{4}\theta_1|k|(t-s)}e^{z\jbr{k,ks}^\gamma}\jbr{k,ks}^\sigma\hS(s,k)\d s.
			\end{aligned}
		\end{equation}
		As a result, the desired estimate (\ref{estFr}) follows from (\ref{T34e7}).
	\end{proof}
	
	\begin{remark}
		\label{R1}
		In the proof of theorem \ref{linearthm}, we do not use the relation that $\hS(t,k)=\hf_+(k,kt)-\hf_-(k,kt)$, so (\ref{estFr}) is valid for any $\rho_\pm$ and $S_\pm$ which satisfy the equation (\ref{lde1}) and (\ref{lde2}). This is useful in the proof of non-linear Landau damping.
	\end{remark}
	
	\section{Nonlinear Landau damping}
	Let us recall the non-linear Vlasov-Poisson system (\ref{eq}) as follows: 
	\begin{equation}
		\label{eqs4}
		\left\{
		\begin{aligned}
			&\partial_t g_++\varepsilon E(t,x+vt) \nabla_v \mu(v) =-\varepsilon E(t,x+vt)(\nabla_v-t\nabla_x)g_+,\\
			&\partial_t g_--E(t,x+vt) \nabla_v \mu(v)= E(t,x+vt)(\nabla_v-t\nabla_x)g_-,\\
			&E=-\nabla_x \phi, \ -\Delta_x \phi=\rho=\rho_+-\rho_-,\\
			& \rho_\pm=\int_{\R^d}g_\pm(t,x-vt,v)\d v.
		\end{aligned}
		\right.
	\end{equation}
	In this section, we first show an inequality
	\[\partial_t G\left[g(t)\right](z)\leqslant C F[\rho](t,z)G\left[g(t)\right]^{\frac{d}{d+1}}(t,z)+C(1+t)F[\rho](t,z)\partial_z G\left[g(t)\right](z).\]
	Then we use the bootstrap method to prove the Landau damping of non-linear Vlasov-Poisson system. In this section, we set the Gevrey index $\gamma=1$ due to the techniques of the proof.
	
	\subsection{Estimate of $G\left[g(t)\right](z)$}
	\begin{lemma}
		\label{L41}
		For $t\geqslant 0$, $z\leqslant \frac{1}{2}\theta_1$, there exists a constant $C$ such that
		\begin{equation}
			\label{Ggtz}
			\partial_t \G{g(t)}(z)\leqslant C F[\rho](t,z)G\left[g(t)\right]^{\frac{d}{d+1}}(t,z)+C(1+t)F[\rho](t,z)\partial_z G\left[g(t)\right](z).
		\end{equation}
	\end{lemma}
	
	\begin{proof}
		Taking the Fourier transform over (\ref{eqs4}) with respect to $x,v$, and also noting that $\hE(t,0)=0$, we have
		\begin{equation}
			\label{L41e1}
			\partial_t \hg_+(t,k,\eta)+\varepsilon\hE(t,k)\hnmu(\eta-kt)=-\frac{\varepsilon}{(2\pi)^d}i\left[\sum_{l\in \Z^d\setminus  \{0\}}(\eta-kt)\hE(t,l)\hg_+(t,k-l,\eta-lt)\right],
		\end{equation}
		\begin{equation}
			\label{L41e2}
			\partial_t \hg_-(t,k,\eta)-\hE(t,k)\hnmu(\eta-kt)=\frac{1}{(2\pi)^d}i\left[\sum_{l\in \Z^d\setminus \{0\}}(\eta-kt)\hE(t,l)\hg_-(t,k-l,\eta-lt)\right].
		\end{equation}
In what follows our proof will be based on equations (\ref{L41e1}) and (\ref{L41e2}).
		
		For simplicity, we rewrite the generator functions (\ref{Fdef}) and (\ref{Gdef}) by the notation $A_{k,\eta}=e^{z\jbr{k,\eta}^\gamma} \jbr{k,\eta}^\sigma$ as
		\begin{equation}
			\label{Fdef2}
			F[\rho](t,z)=\sup_{k\in \Z^d\setminus\{0\}}A_{k,\eta} |\hat{\rho}(t,k)| |k|^{-\alpha},
		\end{equation}
		and
		\begin{equation*}
			G[g(t)](z)=\sum_{|j|\leqslant 1}\sum_{k\in\Z^d}\int_{\R^d} A_{k,\eta}^{d+1}\left[ \left|\partial^j_\eta \hat{g}_+(k,\eta)\right|^{d+1}+\left|\partial^j_\eta \hat{g}_-(k,\eta)\right|^{d+1} \right]  \d \eta.
		\end{equation*}
		Now, the left-hand-side of (\ref{Ggtz}) becomes
		\begin{equation}
			\label{Gtdef}
			\partial_t \G{g(t)}(z)=\sum_{|j|\leqslant 1}\sum_{k\in\Z^d}\int_{\R^d} A_{k,\eta}^{d+1}\left[\partial_t \left|\partial_\eta^j \hg_+(t,k,\eta)\right|^{d+1}+\partial_t\left|\partial_\eta^j \hg_-(t,k,\eta)\right|^{d+1}\right] \d \eta.
		\end{equation}
		To prove the target lemma, we also need some universal inequalities. Thanks to \cite[Equations (4.2,4.3)]{GNR}, we have
		\begin{equation}
			\label{GNR42}
			\begin{aligned}
				\jbr{k,\eta} &\leqslant 2\jbr{k',\eta'}\jbr{k-k',\eta-\eta'}\\
				\jbr{k,\eta} &\leqslant \jbr{k',\eta'}+\jbr{k-k',\eta-\eta'}\\
				A_{k,\eta} &\leqslant C A_{k',\eta'}A_{k-k',\eta-\eta'}
			\end{aligned}
		\end{equation}
		for some universal constant $C$. Now, let's begin to estimate (\ref{Gtdef}).
		
		Firstly, we consider the case $j=0$. Here $j$ is the differential index. Let
		\begin{equation}
			\label{G0def}
			G_0=\sum_{k\in\Z^d}\int_{\R^d}\left[\partial_t \left|\hAg_+(t,k,\eta)\right|^{d+1}+\partial_t\left|\hAg_-(t,k,\eta)\right|^{d+1}\right] \d \eta
		\end{equation}
		Here $\hAg(t,k,\eta)=A_{k,\eta}\hg(t,k,\eta)$. By direct computation,
		\begin{equation}
			\label{Ag12t}
			\partial_t \abs{\hAg_\pm(t,k,\eta)}^{d+1}=(d+1)\abs{\hAg_\pm(t,k,\eta)}^{d-1}\Re \left[\overline{\hAg}_\pm(t,k,\eta)\partial_t \hAg_\pm(t,k,\eta)\right].
		\end{equation}
		Use (\ref{L41e1}) and (\ref{L41e2}), the differential term in (\ref{Ag12t}) is
		\begin{equation}
			\label{Ag1t}
			\begin{aligned}
				\partial_t\hAg_+(t,k,\eta)=&-\varepsilon A_{k,\eta}\hE(t,k)\hnmu(\eta-kt)-\frac{\varepsilon}{(2\pi)^d}iA_{k,\eta}\sum_{l\in \Z^d\setminus \{0\}}(\eta-kt)\hE(t,l)\hg_+(t,k-l,\eta-lt),
			\end{aligned}
		\end{equation}
		\begin{equation}
			\label{Ag2t}
			\begin{aligned}
				\partial_t\hAg_-(t,k,\eta)=& A_{k,\eta}\hE(t,k)\hnmu(\eta-kt)+\frac{1}{(2\pi)^d}iA_{k,\eta}\sum_{l\in \Z^d\setminus \{0\}}(\eta-kt)\hE(t,l)\hg_-(t,k-l,\eta-lt).
			\end{aligned}
		\end{equation}
		Now we combine (\ref{G0def}), (\ref{Ag1t}), (\ref{Ag2t}) and (\ref{Ag12t}). We can write
		\begin{equation}
			\label{G0}
			\begin{aligned}
				G_0=&-(d+1)\varepsilon\sum_{k\in\Z^d\setminus \{0\}}\int_{\R^d}\abs{\hAg_+(t,k,\eta)}^{d-1}\Re\left[A_{k,\eta}\hE(t,k)\hnmu(\eta-kt)\overline{\hAg}_+(t,k,\eta)\right] \d\eta\\
				&-\frac{(d+1)\varepsilon}{(2\pi)^d}\sum_{k\in\Z^d}\int_{\R^d}\abs{\hAg_+(t,k,\eta)}^{d-1} \\ 
				&\Re \left[i(\eta-kt)A_{k,\eta}\sum_{l\in \Z^d\setminus \{0\}}\hE(t,l)\hg_+(t,k-l,\eta-lt)\overline{\hAg}_+(t,k,\eta) \right] \d \eta\\
				&+(d+1)\sum_{k\in\Z^d\setminus \{0\}}\int_{\R^d}\abs{\hAg_-(t,k,\eta)}^{d-1}\Re\left[A_{k,\eta}\hE(t,k)\hnmu(\eta-kt)\overline{\hAg}_-(t,k,\eta)\right] \d\eta \\
				&+\frac{d+1}{(2\pi)^d}\sum_{k\in\Z^d}\int_{\R^d}\abs{\hAg_-(t,k,\eta)}^{d-1}\\
				&\Re \left[i(\eta-kt)A_{k,\eta}\sum_{l\in \Z^d\setminus \{0\}}\hE(t,l)\hg_-(t,k-l,\eta-lt)\overline{\hAg}_-(t,k,\eta)\right] \d \eta \\
				=& -(d+1)\varepsilon G_0^1-\frac{(d+1)\varepsilon}{(2\pi)^d}G_0^2+(d+1)G_0^3+\frac{d+1}{(2\pi)^d}G_0^4.
			\end{aligned}
		\end{equation}
		Now, we are going to estimate $G_0^1\sim G_0^4$ in (\ref{G0}).
		
		{\bf $G_0^1$ and $G_0^3$:}
		
		Recall (\ref{GNR42}), we have
		\begin{equation}\label{G013}
			\begin{aligned}
				& \left|\sum_{k\in\Z^d\setminus \{0\}}\int_{\R^d}\abs{\hAg_\pm(t,k,\eta)}^{d-1} \Re\left[A_{k,\eta}\hE(t,k)\hnmu(\eta-kt)\overline{\hAg_\pm}(t,k,\eta)\right] \d\eta\right|\\
				\leqslant & C\sum_{k\in\Z^d\setminus \{0\}}\int_{\R^d} A_{k,kt}A_{0,\eta-kt}\left|\hE(t,k)\right|\jbr{\eta-kt}\left|\hmu(\eta-kt)\right|\left|\hAg_\pm(t,k,\eta)\right|^d \d \eta.
			\end{aligned}
		\end{equation}
		By the third equation of (\ref{eqs4}), we have that $\hE(t,k)=-\frac{i}{\abs{k}^2}k\cdot \hrho(t,k)$. Also,  $A_{0,\eta-kt}=\jbr{\eta-kt}^\sigma e^{z\jbr{\eta-kt}}$ by definition and $\hmu(\eta-kt)\leqslant C e^{-\theta_0\jbr{\eta-kt}}$ by (\ref{H1}),
		\begin{equation}
			\label{G013n1}
			\begin{aligned}
				&\sum_{k\in\Z^d\setminus \{0\}}\int_{\R^d} A_{k,kt}A_{0,\eta-kt}\left|\hE(t,k)\right|\jbr{\eta-kt}\left|\hmu(\eta-kt)\right|\left|\hAg_\pm(t,k,\eta)\right|^d \d \eta \\
				\leqslant & C\sum_{k\in\Z^d\setminus \{0\}} \frac{A_{k,kt}}{|k|}\left|\hrho(t,k)\right|\int_{\R^d}\jbr{\eta-kt}^{1+\sigma} e^{z\jbr{\eta-kt}-\theta_0\jbr{\eta-kt}} \left|\hAg_\pm(t,k,\eta)\right|^d\d \eta.
			\end{aligned}
		\end{equation}
		Using H\"{o}lder's inequality for sums,
		\begin{equation}
			\label{G013n2}
			\begin{aligned}
				&\sum_{k\in\Z^d\setminus \{0\}} \frac{A_{k,kt}}{|k|}\left|\hrho(t,k)\right|\int_{\R^d}\jbr{\eta-kt}^{1+\sigma} e^{z\jbr{\eta-kt}-\theta_0\jbr{\eta-kt}} \left|\hAg_\pm(t,k,\eta)\right|^d\d \eta\\
				\leqslant & C\left(\sum_{k\in\Z^d\setminus\{0\}} \left(\frac{A_{k,kt}}{|k|}\right)^{d+1}\left|\hrho(t,k)\right|^{d+1}\right)^{\frac{1}{d+1}}\times \\
				& \left(\sum_{k\in\Z^d\setminus\{0\}}\left(\int_{\R^d}\jbr{\eta-kt}^{1+\sigma} e^{z\jbr{\eta-kt}-\theta_0\jbr{\eta-kt}} \left|\hAg_\pm(t,k,\eta)\right|^d\d \eta\right)^{\frac{d+1}{d}}\right)^{\frac{d}{d+1}}.
			\end{aligned}
		\end{equation}
		Now let's consider the two terms in (\ref{G013n2}) separately. Regarding the first term, also note that $\alpha<\frac{1}{d+1}$, we have
		\begin{equation}
			\label{G0131}
			\begin{aligned}
				&\sum_{k\in\Z^d\setminus \{0\}}\left(\frac{A_{k,kt}}{|k|}\left|\hrho(t,k)\right|\right)^{d+1}\\
				= &\sum_{k\in\Z^d\setminus \{0\}}\left(A_{k,kt}\left|\hrho(t,k)\right||k|^{-\alpha}\right)^{d+1}\frac{1}{|k|^{d+1-(d+1)\alpha}}\\
				\leqslant &F[\rho]^{d+1}(t,z)\sum_{k\in\Z^d\setminus \{0\}}\frac{1}{|k|^{d+1-(d+1)\alpha}} \\
				\leqslant &CF[\rho]^{d+1}(t,z).
			\end{aligned}
		\end{equation}
		For the second term in (\ref{G013n2}), we apply the H\"{o}lder's inequality,
		\begin{equation}
			\label{G0132}
			\begin{aligned}
				&\sum_{k\in\Z^d\setminus \{0\}}\left(\int_{\R^d}\jbr{\eta-kt}^{1+\sigma} e^{z\jbr{\eta-kt}-\theta_0\jbr{\eta-kt}} \left|\hAg_\pm(t,k,\eta)\right|^d\d \eta\right)^{\frac{d+1}{d}}\\
				\leqslant & \sum_{k\in\Z^d\setminus \{0\}}\left(\int_{\R^d}\left(\jbr{\eta-kt}^{1+\sigma} e^{z\jbr{\eta-kt}-\theta_0\jbr{\eta-kt}}\right)^{d+1}\d \eta\right)^{\frac{1}{d}}\left(\int_{\R^d}\abs{\hAg_\pm(t,k,\eta)}^{d+1}\d \eta\right)\\
				\leqslant & \sum_{k\in\Z^d\setminus \{0\}}\left(\int_{\R^d}\left(\jbr{\eta}^{1+\sigma} e^{z\jbr{\eta}-\theta_0\jbr{\eta}}\right)^{d+1}\d \eta\right)^{\frac{1}{d}}\left(\int_{\R^d}\abs{\hAg_\pm(t,k,\eta)}^{d+1}\d \eta\right) \\
				\leqslant & C\sum_{k\in\Z^d\setminus \{0\}} \left(\int_{\R^d}\abs{\hAg_\pm(t,k,\eta)}^{d+1}\d \eta\right)\\
				\leqslant & C\G{g(t)}(z).
			\end{aligned}
		\end{equation}
		The third inequality in (\ref{G0132}) used that $z\leqslant \frac{\theta_1}{2}<\frac{\theta_0}{2}$. Now, we put (\ref{G013}), (\ref{G013n1}), (\ref{G013n2}), (\ref{G0131}) and (\ref{G0132}) together, we deduce
		\begin{equation}
			\label{G013est}
			G_0^1,G_0^3\leqslant CF[\rho](t,z)\G{g(t)}^{\frac{d}{d+1}}(z).
		\end{equation}
		
		{\bf $G_0^2$ and $G_0^4$:}
		
		To estimate these two terms, we firstly give two useful inequalities,
		\begin{equation}
			\label{ekt1}
			\begin{aligned}
				\abs{\eta-kt} &=\sqrt{\eta^2-2\eta kt+k^2t^2}\\
				&\leqslant\sqrt{\eta^2+\eta^2t^2+k^2+k^2t^2}\\
				&\leqslant \sqrt{\left(1+\eta^2+k^2\right)(1+t^2)}\\
				&=\jbr{k,\eta}\jbr{t},
			\end{aligned}
		\end{equation}
		and similarly,
		\begin{equation}
			\label{ekt2}
			\begin{aligned}
				\abs{\eta-kt}&=\abs{\eta-lt-\left(k-l\right)t}\\
				&\leqslant \jbr{\eta-lt,k-l}\jbr{t}.
			\end{aligned}
		\end{equation}
		Then we can estimate $G_0^2,G_0^4$ as follows:
		\begin{equation}
			\label{G024}
			\begin{aligned}
				&\abs{\sum_{k\in\Z^d}\int_{\R^d}\abs{\hAg_\pm(t,k,\eta)}^{d-1}\Re \left[i\left(\eta-kt\right)A_{k,\eta}\sum_{l\in \Z^d\setminus  \{0\}}\hE(t,l)\hg_\pm (t,k-l,\eta-lt)\overline{\hAg}_\pm(t,k,\eta)\right]\d \eta}\\
				\leqslant &\sum_{k\in\Z^d}\sum_{l\in \Z^d\setminus  \{0\}}\int_{\R^d}\abs{\eta-kt}A_{k,\eta}\abs{\hE(t,l)}\abs{\hg_\pm(t,k-l,\eta-lt)}\abs{\hAg_\pm(t,k,\eta)}^d\d \eta.
			\end{aligned}
		\end{equation}
		Using $\abs{\hE(t,l)}=\abs{l}^{-1}\abs{\hrho(t,l)}$, we have
		\begin{equation}
			\label{G024n1}
			\begin{aligned}
				&\sum_{k\in\Z^d}\sum_{l\in \Z^d\setminus  \{0\}}\int_{\R^d}\abs{\eta-kt}A_{k,\eta}\abs{\hE(t,l)}\abs{\hg_\pm(t,k-l,\eta-lt)}\abs{\hAg_\pm(t,k,\eta)}^d\d \eta\\
				\leqslant &\sum_{k\in\Z^d}\sum_{l\in \Z^d\setminus  \{0\}}\int_{\R^d}\abs{\eta-kt}\frac{A_{k,\eta}}{A_{l,lt}A_{k-l.\eta-lt}}A_{l,lt}\abs{l}^{-1}\abs{\hrho(t,l)}\abs{\hAg_\pm(t,k-l,\eta-lt)}\abs{\hAg_\pm(t,k,\eta)}^d\d \eta.
			\end{aligned}
		\end{equation}
		Using the definition of $F[\rho]$ (\ref{Fdef2}), 
		\begin{equation}
			\label{G024n2}
			\begin{aligned}
				&\sum_{k\in\Z^d}\sum_{l\in \Z^d\setminus  \{0\}}\int_{\R^d}\abs{\eta-kt}\frac{A_{k,\eta}}{A_{l,lt}A_{k-l.\eta-lt}}A_{l,lt}\abs{l}^{-1}\abs{\hrho(t,l)}\abs{\hAg_\pm(t,k-l,\eta-lt)}\abs{\hAg_\pm(t,k,\eta)}^d\d \eta\\
				\leqslant &F[\rho](t,z)\sum_{k\in\Z^d}\sum_{l\in \Z^d\setminus  \{0\}}\int_{\R^d}\abs{\eta-kt}\frac{A_{k,\eta}}{A_{l,lt}A_{k-l.\eta-lt}}\abs{l}^{-1+\alpha}\abs{\hAg_\pm(t,k-l,\eta-lt)}\abs{\hAg_\pm(t,k,\eta)}^d\d \eta.
			\end{aligned}
		\end{equation}
		Then we are going to estimate $\frac{A_{k,\eta}}{A_{l,lt}A_{k-l,\eta-lt}}$. We claim that
		\begin{equation}
			\label{G024cl}
			\frac{A_{k,\eta}}{A_{l,lt}A_{k-l,\eta-lt}}\leqslant C\left(\jbr{k-l}^{-\sigma}+\jbr{l}^{-\sigma}\right).
		\end{equation}
		We prove the claim by considering two cases. The second inequality in (\ref{GNR42}) indicates that either $\jbr{l,lt}\geqslant \frac{1}{2}\jbr{k,\eta}$ or $\jbr{k-l,\eta-lt}\geqslant \frac{1}{2}\jbr{k,\eta}$.
		
		{\bf Case 1: } If $\jbr{l,lt}\geqslant \frac{1}{2}\jbr{k,\eta}$.
		
		Use (\ref{GNR42}), we have
		\begin{equation}
			\label{G024c1}
			\begin{aligned}
				\frac{A_{k,\eta}}{A_{l,lt}A_{k-l,\eta-lt}}&=\frac{e^{z\jbr{k,\eta}}\jbr{k,\eta}^\sigma}{\jbr{l,lt}^\sigma\jbr{k-l,\eta-lt}^\sigma e^{z\jbr{l,lt}+z\jbr{k-l,\eta-lt}}}\\
				&\leqslant\frac{\jbr{k,\eta}^\sigma}{\jbr{l,lt}^\sigma\jbr{k-l,\eta-lt}^\sigma}\\
				&\leqslant\frac{C}{\jbr{k-l,\eta-lt}^\sigma}\\
				&\leqslant C\jbr{k-l}^{-\sigma}.
			\end{aligned}
		\end{equation}
		
		{\bf Case 2: } If $\jbr{k-l,\eta-lt}\geqslant \frac{1}{2}\jbr{k,\eta}$.
		
		Similarly to case 1,
		\begin{equation}
			\label{G024c2}
			\begin{aligned}
				\frac{A_{k,\eta}}{A_{l,lt}A_{k-l,\eta-lt}}
				&\leqslant\frac{\jbr{k,\eta}^\sigma}{\jbr{l,lt}^\sigma\jbr{k-l,\eta-lt}^\sigma}\\
				&\leqslant\frac{C}{\jbr{l,lt}^\sigma}\\
				&\leqslant C\jbr{l}^{-\sigma}.
			\end{aligned}
		\end{equation}
		Combining (\ref{G024c1}) and(\ref{G024c2}), we complete the proof of (\ref{G024cl}).
		
		Now let's continue the estimate in (\ref{G024n2}). By (\ref{ekt1}) and (\ref{ekt2}), we have
		\begin{equation}
			\label{G024ekt}
			\abs{\eta-kt}\leqslant\jbr{t}\jbr{k,\eta}^{\frac{d}{d+1}}\jbr{k-l,\eta-lt}^{\frac{1}{d+1}}.
		\end{equation}
		Combining (\ref{G024cl}) and (\ref{G024ekt}),
		\begin{equation}
			\label{G0241}
			\begin{aligned}
				&F[\rho](t,z)\sum_{k\in\Z^d}\sum_{l\in \Z^d\setminus  \{0\}}\int_{\R^d}\abs{\eta-kt}\frac{A_{k,\eta}}{A_{l,lt}A_{k-l.\eta-lt}}\abs{l}^{-1+\alpha}\abs{\hAg_\pm(t,k-l,\eta-lt)}\abs{\hAg_\pm(t,k,\eta)}^d\d \eta \\
				\leqslant & C \jbr{t}F[\rho](t,z)\sum_{k\in\Z^d}\sum_{l\in \Z^d\setminus  \{0\}}\left(\jbr{l-k}^{-\sigma}+\jbr{l}^{-\sigma}\right)\int_{\R^d} \abs{l}^{-1+\alpha}\jbr{k-l,\eta-lt}^{\frac{1}{d+1}}\abs{\hAg_\pm(t,k-l,\eta-lt)} \\
				&\jbr{k,\eta}^{\frac{d}{d+1}}\abs{\hAg_\pm(t,k,\eta)}^d\d \eta .
			\end{aligned}
		\end{equation}
		Young's inequality gives
		\begin{equation}
			\label{Young}
			ab\leqslant \frac{a^{d+1}}{d+1}+\frac{b^{\frac{d+1}{d}}}{\frac{d+1}{d}}\leqslant a^{d+1}+b^{\frac{d+1}{d}}
		\end{equation}
		for any $a,b>0$, so
		\begin{equation}
			\label{G024n3}
			\begin{aligned}
				& \sum_{k\in\Z^d}\sum_{l\in \Z^d\setminus  \{0\}}\left(\jbr{l-k}^{-\sigma}+\jbr{l}^{-\sigma}\right)\int_{\R^d} \abs{l}^{-1+\alpha}\jbr{k-l,\eta-lt}^{\frac{1}{d+1}}\abs{\hAg_\pm(t,k-l,\eta-lt)} \\
				&\jbr{k,\eta}^{\frac{d}{d+1}}\abs{\hAg_\pm(t,k,\eta)}^d\d \eta \\
				\leqslant & \sum_{k\in\Z^d}\sum_{l\in \Z^d\setminus  \{0\}}\left(\jbr{l-k}^{-\sigma}+\jbr{l}^{-\sigma}\right)\bigg[\int_{\R^d} \abs{l}^{-(d+1)+(d+1)\alpha}\jbr{k-l,\eta-lt}\\
				&\abs{\hAg_\pm(t,k-l,\eta-lt)}^{d+1}\d \eta+\int_{\R^d}\jbr{k,\eta}\abs{\hAg_\pm(t,k,\eta)}^{d+1}\d\eta\bigg]\\
				\leqslant &\biggl\{\sum_{k\in\Z^d}\sum_{l\in \Z^d\setminus  \{0\}}\int_{\R^d} \abs{l}^{-(d+1)+(d+1)\alpha}\jbr{k-l,\eta-lt}\abs{\hAg_\pm(t,k-l,\eta-lt)}^{d+1}\d \eta\\
				&+ \sum_{k\in\Z^d}\sum_{l\in \Z^d\setminus  \{0\}}\jbr{l-k}^{-\sigma} \int_{\R^d}\jbr{k,\eta}\abs{\hAg_\pm(t,k,\eta)}^{d+1}\d\eta\\
				&+\sum_{k\in\Z^d}\sum_{l\in \Z^d\setminus  \{0\}}\jbr{l}^{-\sigma} \int_{\R^d}\jbr{k,\eta}\abs{\hAg_\pm(t,k,\eta)}^{d+1}\d\eta\biggr\}.
			\end{aligned}
		\end{equation}
		Now we are going to estimate (\ref{G024n3}). Note that $\alpha<\frac{1}{d+1}$, the first integration term becomes
		\begin{equation}
			\label{G024i1}
			\begin{aligned}
				&\sum_{k\in\Z^d}\sum_{l\in \Z^d\setminus  \{0\}}\abs{l}^{-(d+1)+(d+1)\alpha}\int_{\R^d}\jbr{k-l,\eta-lt}\abs{\hAg_\pm(t,k-l,\eta-lt)}^{d+1}\d \eta\\
				\leqslant &\sum_{l\in \Z^d\setminus  \{0\}}\abs{l}^{-(d+1)+(d+1)\alpha}\sum_{k-l\in\Z^d} \int_{\R^d}\jbr{k-l,\eta}\abs{\hAg_\pm(t,k-l,\eta)}^{d+1}\d \eta\\
				\leqslant &C \partial_z\G{g(t)}(z).
			\end{aligned}
		\end{equation}
		Then, for the second integration term in (\ref{G024n3}), by $\sigma>d+1$, let $l=n+k$,
		\begin{equation}
			\label{G024i2}
			\begin{aligned}
				&\sum_{k\in\Z^d}\sum_{l\in \Z^d\setminus  \{0\}}\jbr{l-k}^{-\sigma} \int_{\R^d}\jbr{k,\eta}\abs{\hAg_\pm(t,k,\eta)}^{d+1}\d\eta\\
				\leqslant & \sum_{k\in\Z^d}\sum_{n\in \Z^d} \jbr{n}^{-\sigma} \int_{\R^d}\jbr{k,\eta}\abs{\hAg_\pm(t,k,\eta)}^{d+1}\d\eta\\
				\leqslant & C\partial_z\G{g(t)}(z).
			\end{aligned}
		\end{equation}
		Similarly, for the third integration term in (\ref{G024n3}),
		\begin{equation}
			\label{G024i3}
			\begin{aligned}
				&\sum_{k\in\Z^d}\sum_{l\in \Z^d\setminus  \{0\}}\jbr{l}^{-\sigma} \int_{\R^d}\jbr{k,\eta}\abs{\hAg_\pm(t,k,\eta)}^{d+1}\d\eta\\
				\leqslant &C\partial_z\G{g(t)}(z).
			\end{aligned}
		\end{equation}
		Combining (\ref{G024n3}), (\ref{G024i1}), (\ref{G024i2}) and (\ref{G024i3}),
		\begin{equation}
			\label{G0242}
			\begin{aligned}
				& \sum_{k\in\Z^d}\sum_{l\in \Z^d\setminus  \{0\}}\left(\jbr{l-k}^{-\sigma}+\jbr{l}^{-\sigma}\right)\int_{\R^d} \abs{l}^{-1+\alpha}\jbr{k-l,\eta-lt}^{\frac{1}{d+1}}\abs{\hAg_\pm(t,k-l,\eta-lt)} \\
				&\jbr{k,\eta}^{\frac{d}{d+1}}\abs{\hAg_\pm(t,k,\eta)}^d\d \eta \\
				\leqslant &C\partial_z\G{g(t)}(z).
			\end{aligned}
		\end{equation}
		By combining (\ref{G024}), (\ref{G024n1}), (\ref{G024n2}), (\ref{G0241}), and (\ref{G0242}), we obtain the following expression:
		\begin{equation}
			\label{G024est}
			G_0^2,G_0^4\leqslant C\jbr{t}F[\rho](t,z)\partial_z\G{g(t)}(z).
		\end{equation}
		Now, we finish the estimate of $G_0^1\sim G_0^4$ (\ref{G013est},\ref{G024est}), which means by (\ref{G0}),
		\begin{equation}
			\label{G0est}
			G_0\leqslant CF[\rho](t,z)\G{g(t)}^{\frac{d}{d+1}}(z)+C(1+t)F[\rho](t,z)\partial_z\G{g(t)}(z).
		\end{equation}
		
		Secondly, let's consider the integral term in (\ref{Gtdef}) involving $\abs{\partial_t\partial_\eta^j\hg_\pm(t,k,\eta)}$ where $\abs{j}=1$. By (\ref{L41e1}) and (\ref{L41e2}), it holds that
		\begin{equation}
			\label{L41j1}
			\begin{aligned}
				\partial_t\partial_\eta^j\hg_+(t,k,\eta)=&-\varepsilon \hE(t,k)\partial_\eta^j \widehat{\nabla_v\mu}(\eta-kt)-\frac{\varepsilon}{(2\pi)^d}i\sum_{l\in \Z^d\setminus \{0\}}(\eta-kt)\hE(t,l)\partial_\eta^j\hg_+(t,k-l,\eta-lt)\\
				&-\frac{\varepsilon}{(2\pi)^d} i\sum_{l\in \Z^d\setminus  \{0\}}\hE_j(t,l)\hg_+(t,k-l,\eta-lt),
			\end{aligned}
		\end{equation}
		and
		\begin{equation}
			\label{L41j2}
			\begin{aligned}
				\partial_t\partial_\eta^j\hg_-(t,k,\eta)=& \hE(t,k)\partial_\eta^j \widehat{\nabla_v\mu}(\eta-kt)+\frac{1}{(2\pi)^d}i\sum_{l\in \Z^d\setminus \{0\}}(\eta-kt)\hE(t,l)\partial_\eta^j\hg_-(t,k-l,\eta-lt)\\
				&+\frac{1}{(2\pi)^d} i\sum_{l\in \Z^d\setminus  \{0\}}\hE_j(t,l)\hg_-(t,k-l,\eta-lt).
			\end{aligned}
		\end{equation}
		Here $\hE_j(t,l)=\partial_\eta^j\left((\eta-kt)\hE(t,l)\right)$. The first two terms are treated similarly for the case $j=0$, and we can get the same conclusion as in (\ref{G0est}). So we only need to pay attention to the third term in (\ref{L41j1}) and (\ref{L41j2}). Let $\Aetag{\pm}(t,k,\eta)=A_{k,\eta}\partial_\eta^j\hg_\pm(t,k,\eta)$, then it holds that
		\begin{equation}
			\label{G1}
			A_{k,\eta}^{d+1}\partial_t\abs{\partial_\eta^j\hg_\pm(t,k,\eta)}^{d+1}=(d+1)\abs{\Aetag{\pm}(t,k,\eta)}^{d-1}\Re\left[\overline{\Aetag{\pm}}(t,k,\eta)\partial_t\Aetag{\pm}(t,k,\eta)\right].
		\end{equation} 
		As we have discussed before, we only take the third term of (\ref{L41j1}) and (\ref{L41j2}) into (\ref{G1}), so we only need to estimate
		\begin{equation}
			\label{N2}
			\begin{aligned}
				N_2=&\abs{\sum_{k\in\Z^d}\int_{\R^d}\abs{\Aetag{\pm}(t,k,\eta)}^{d-1}\Re\left[\overline{\Aetag{\pm}}(t,k,\eta)\sum_{l\in \Z^d\setminus  \{0\}}A_{k,\eta}\hE_j(t,l)\hg_\pm(t,k-l,\eta-lt)\right]}\\
				\leqslant &\sum_{k\in\Z^d}\sum_{l\in \Z^d\setminus  \{0\}}\int_{\R^d}A_{k,\eta}\abs{\hE_j(t,l)}\abs{\Aetag{\pm}(t,k,\eta)}^d\abs{\hg_\pm(t,k-l,\eta-lt)}\\
				\leqslant &\sum_{k\in\Z^d}\sum_{l\in \Z^d\setminus \{0\}}\int_{\R^d}A_{l,lt}\abs{\hrho(t,l)}\abs{l}^{-\alpha}\abs{l}^{-1+\alpha}\frac{A_{k,\eta}}{A_{l,lt}A_{k-l,\eta-lt}}\abs{\hAg_\pm(t,k-l,\eta-lt)}\abs{\Aetag{\pm}(t,k,\eta)}^d\d \eta\\
				\leqslant & F[\rho](t,z)\sum_{k\in\Z^d}\sum_{l\in \Z^d\setminus \{0\}}\int_{\R^d} \frac{A_{k,\eta}}{A_{l,lt}A_{k-l,\eta-lt}}\abs{l}^{-1+\alpha}\abs{\hAg_\pm(t,k-l,\eta-lt)}\abs{\Aetag{\pm}(t,k,\eta)}^d\d \eta.
			\end{aligned}
		\end{equation}
		We use again (\ref{G024cl}), together with (\ref{Young}), under the same analysis as in (\ref{G024i1}), (\ref{G024i2}) and (\ref{G024i3}), we have
		\begin{equation}
			\label{N21}
			\begin{aligned}
				&\sum_{k\in\Z^d}\sum_{l\in \Z^d\setminus \{0\}}\int_{\R^d} \frac{A_{k,\eta}}{A_{l,lt}A_{k-l,\eta-lt}}\abs{l}^{-1+\alpha}\abs{\hAg_\pm(t,k-l,\eta-lt)}\abs{\Aetag{\pm}(t,k,\eta)}^d\d \eta\\
				\leqslant &\sum_{k\in\Z^d}\sum_{l\in \Z^d\setminus  \{0\}}\left(\jbr{l-k}^{-\sigma}+\jbr{l}^{-\sigma}\right)\\
				&\left(\int_{\R^d}\abs{l}^{-(d+1)+(d+1)\alpha}\abs{\hAg_\pm(t,k-l,\eta-kl)}^{d+1}\d \eta+\int_{\R^d}\abs{\Aetag{\pm}(t,k,\eta)}^{d+1}\d \eta\right)\\
				\leqslant &C\G{g(t)}(z)\\
				\leqslant &C\partial_z \G{g(t)}(z).
			\end{aligned}
		\end{equation}
		In the last inequality, we use the fact that 
		$$
		\partial_zA_{k,\eta}^{d+1}=(d+1)\jbr{k,\eta}A_{k,\eta}^{d+1}>A_{k,\eta}^{d+1}.
		$$
		By (\ref{N2}) and (\ref{N21}), we have
		\begin{equation}
			\label{N2est}
			N_2\leqslant CF[\rho](t,z)\partial_z\G{g(t)}(z).
		\end{equation}
		Now, let's combine the estimates in two cases $j=0$ (\ref{G0est}) and $|j|=1$ (\ref{N2est}), we are able to prove (\ref{Ggtz}). Now we finish the proof of Lemma \ref{L41}.
	\end{proof}
	
	\subsection{Two priori estimates}
	Now, we are going to prove the Landau damping of (\ref{eqs4}). Let $\lambda(t)=\lambda_0+\lambda_0(1+t)^{-\delta}$ for some positive constant $\lambda_0$ such that $0<\lambda_0\leqslant 1$ and $\lambda_0\leqslant \min\{\frac{\lambda_1}{4},\frac{\theta_1}{4}\}$, where $\lambda_1$ is defined in (\ref{Ginitial}) and $\theta_1$ is defined in Lemma \ref{esttK}. $\delta$ is an arbitrary small constant $\delta<1$. As in Theorem \ref{mainthm}, $\sigma>\max\{d+1,3\}$.

	For $C_1,C_2>1$ large enough and $\varepsilon_0$ small enough, we give two propositions which will be proved in next two sections. 
	
	\begin{lemma}
		\label{L1}
		Assume that for any $0\leqslant t\leqslant T$,
		\begin{equation}
			\label{L1e1}
			F[\rho](t,\lambda(t))\leqslant4C_1\varepsilon_0^{\frac{1}{d+1}}\jbr{t}^{-\sigma+1}.
		\end{equation}
		Then it holds that
		\begin{equation}
			\label{L1e2}
			\sup_{0\leqslant t\leqslant T}\G{g(t)}(\lambda(t))\leqslant 4C_2\varepsilon_0.
		\end{equation}
	\end{lemma}
	
	\begin{lemma}
		\label{L2}
		Suppose
		\begin{equation}
			\label{L2e1}
			\sup_{0\leqslant t\leqslant T}\G{g(t)}(\lambda(t))\leqslant 4C_2\varepsilon_0.
		\end{equation}
		Then for any $0\leqslant t\leqslant T$,
		\begin{equation}
			\label{L2e2}
			F[\rho](t,\lambda(t))\leqslant2C_1\varepsilon_0^{\frac{1}{d+1}}\jbr{t}^{-\sigma+1}.
		\end{equation}
	\end{lemma}
	
	By these two priori estimates, we can prove the following result.
	
	\begin{prop}
		\label{p4}
		Assume the conditions stated in Theorem \ref{mainthm}, and
		\begin{equation}
			\label{con1}
			\tag{Assumption 1}
			C_0\leqslant 2C_1,
		\end{equation}
		where $C_0$ is the constant in Lemma \ref{FC0G}. Then the following holds:
		\begin{equation}
			\label{p4eqn}
			\G{g(t)}(\lambda(t))\leqslant 4C_2\varepsilon_0
		\end{equation}
		for any $t>0$.
	\end{prop}
	
	\begin{proof}
		By lemma \ref{FC0G}, (\ref{Ginitial}) and (\ref{con1}), we have
		\begin{equation}
			\label{Fr0}
			F[\rho](0,\lambda(0))\leqslant C_0\G{g(0)}^{\frac{1}{d+1}}(\lambda_1)<4C_1\varepsilon_0^{\frac{1}{d+1}}.
		\end{equation}
		There exists $T>0$ such that (\ref{L1e1}) holds for any $t\in [0,T)$. We define
		\begin{equation*}
			T_*=\sup\left\{T>0:F[\rho](t,\lambda(t))\leqslant4C_1\varepsilon_0^{\frac{1}{d+1}}\jbr{t}^{-\sigma+1} \text{ for any } t\in [0,T)\right\}.
		\end{equation*}
		By (\ref{Fr0}), it is obvious to see that $T_*>0$. 
		
		We claim that $T_*=+\infty$. Otherwise, one has $0<T_*<+\infty$. By the continuity of $F[\rho](t,\lambda(t))$, we have that
		\begin{equation}
			\label{p41}
			F[\rho](T_*,\lambda(T_*))=4C_1\varepsilon_0^{\frac{1}{d+1}}\jbr{T_*}^{-\sigma+1}.
		\end{equation}
		By Lemma \ref{L1},
		\begin{equation*}
			\sup_{0\leqslant t\leqslant T_*}\G{g(t)}(\lambda(t))\leqslant 4C_2\varepsilon_0.
		\end{equation*}
		Then by Lemma \ref{L2}, there holds
		\begin{equation*}
			F[\rho](T_*,\lambda(T_*))\leqslant 2C_1\varepsilon_0^{\frac{1}{d+1}}\jbr{T_*}^{-\sigma+1},
		\end{equation*}
		which is contradicted to (\ref{p41}).
		
		Hence, $T_*=+\infty$. Finally, using Lemma \ref{L1}, we then deduce (\ref{p4eqn}).
	\end{proof}
	
	In the next two sections, we will prove lemma \ref{L1} and \ref{L2}.
	
	\subsection{Proof of Lemma \ref{L1}}
	Define
	\begin{equation}
		\label{tGFdef}
		\tG{g(t)}(z)=\G{g(t)}(\lambda(t)z), \ \tF{\rho}(t,z)=F[\rho](t,\lambda(t)z).
	\end{equation}
	We note that
	\begin{equation*}
		\partial_t \tG{g(t)}(z)=\partial_t \G{g(t)}(\lambda(t)z)+\lambda'(t)z\partial_z \G{g(t)}(\lambda(t)z).
	\end{equation*}
	Now we use Lemma \ref{L41} to obtain
	\begin{equation*}
		\begin{aligned}
			\partial_t \tG{g(t)}(z)\leqslant & C F[\rho](t,\lambda(t)z)\G{g(t)}^{\frac{d}{d+1}}(\lambda(t)z) +C(1+t)F[\rho](t,\lambda(t)z)\partial_z \G{g(t)}(\lambda(t)z)\\
			&+\lambda'(t)z\partial_z\G{g(t)}(\lambda(t)z)\\
			\leqslant &C\tF{\rho}(t,z)\tG{g(t)}^{\frac{d}{d+1}}(z)+\partial_z\tG{g(t)}(z) \frac{C(1+t)\tF{\rho}(t,z)+\lambda'(t)z}{\lambda(t)}.
		\end{aligned}
	\end{equation*}
	That is, 
	\begin{equation}
		\label{pttG}
		\partial_t \tG{g(t)}(z)-\partial_z\tG{g(t)}(z) \frac{C(1+t)\tF{\rho}(t,z)+\lambda'(t)z}{\lambda(t)}\leqslant C\tF{\rho}(t,z)\tG{g(t)}^{\frac{d}{d+1}}(z).
	\end{equation}
	For any $z\in [0,1]$ and any $t\in [0,T]$, the left hand side of (\ref{pttG}) is a transport equation with characteristic ODE:
	\begin{equation}
		\label{ODE}
		\left\{\begin{aligned}
			\dot{Z}(s)&=-\frac{\lambda'(s)Z(s)+C(1+s)\tF{\rho}(s,Z(s))}{\lambda(s)},\\
			Z(t)&=z.
		\end{aligned}
		\right.
	\end{equation}
	
	We firstly prove that $Z(s)\in [0,1]$ for any $s\in[0,t]$ if $Z(t)=z\in[0,1]$. By the first equation of (\ref{ODE}),
	\begin{equation}
		\label{char0}
		\left(\lambda(s)Z(s)\right)'=-C(1+s)\tF{\rho}(s,Z(s))\leqslant 0.
	\end{equation}
	Then for $s\in [0,t]$, $Z(t)=z\in [0,1]$, we take integral over (\ref{char0}),
	\begin{equation}
		\label{chars}
		\lambda(s)Z(s)=\lambda(t)z-\int_{s}^{t}-C(1+\tau)\tF{\rho}(\tau,Z(\tau))\d \tau\geqslant 0.
	\end{equation}
	Next, we are going to prove $Z(s)\leqslant 1$ for $s\in [0,t]$. To prove this result, we only need to show that $Z(s)$ is outgoing if $Z(s)=1$, which means $\dot{Z}(s)>0$ if $Z(s)=1$.
	\begin{figure}[h!]
		\centering
		\includegraphics[scale=0.3]{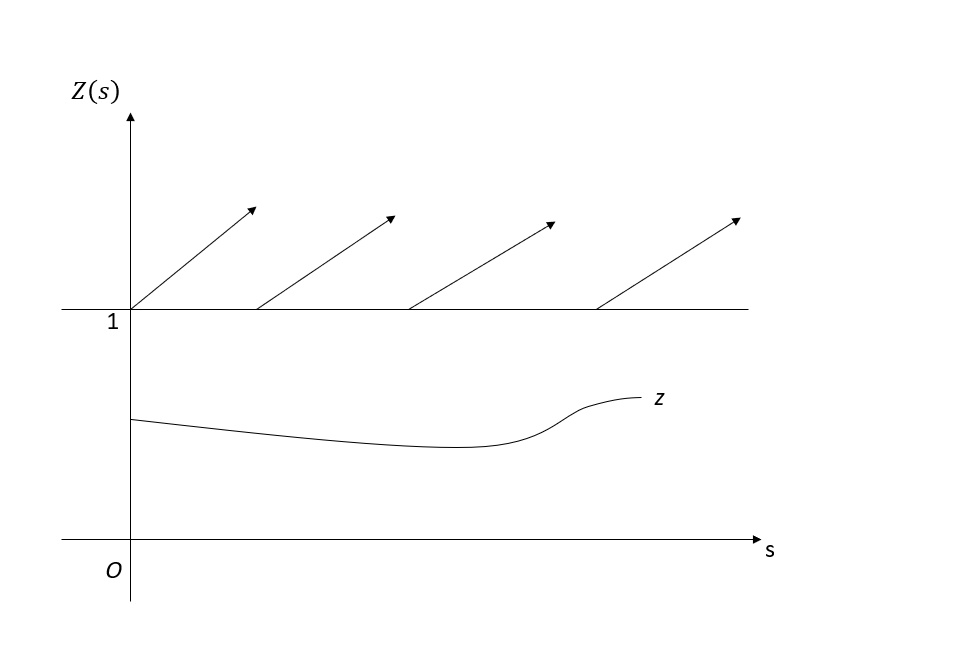}
		\caption{Characteristic curve}
	\end{figure}
	
	\begin{lemma}
		\label{L1l1}
		Assume the conditions stated in Lemma \ref{L1}, and if $Z(s)=1$, we have $\dot{Z}(s)> 0$.
	\end{lemma}
	\begin{proof}
		By (\ref{ODE}),
		\begin{equation}
			\label{dZs}
			\dot{Z}(s)=-\frac{\lambda'(s)+C(1+s)F[\rho](s,\lambda(s))}{\lambda(s)}.
		\end{equation}
		We only need to pay attention to the numerator.
		\begin{equation}
			\label{L1l1e1}
			\begin{aligned}
				\lambda'(s)+C(1+s)F[\rho](s,\lambda(s))&\leqslant -\delta(1+s)^{-\delta-1}\lambda_0+4CC_1(1+s)\varepsilon_0^{\frac{1}{d+1}}(1+s)^{-\sigma+1}\\
				&=-\delta(1+s)^{-\delta-1}\lambda_0+4CC_1\varepsilon_0^{\frac{1}{d+1}}(1+s)^{-\sigma+2}.
			\end{aligned}
		\end{equation}
		Let's name the constant $C$ in the last row of (\ref{L1l1e1}) as $\bar{C}_1$. Since $-\delta-1>-\sigma+2$, which is equivalent to $\delta<\sigma-3$, we have $(1+s)^{-\sigma+2}\leqslant (1+s)^{-\delta-1}$. It then follows
		\begin{equation}
			\label{L1l1e2}
			\lambda'(s)+C(1+s)F[\rho](s,\lambda(s))\leqslant (1+t)^{-\delta-1}\left(-\delta\lambda_0+4\bar{C}_1 C_1\varepsilon_0^{\frac{1}{d+1}}\right)<0,
		\end{equation}
		provided that
		\begin{equation}
			\tag{Assumption 2}
			\label{Lcon2}
			4\bar{C}_1C_1\varepsilon_0^{\frac{1}{d+1}}<\delta\lambda_0.
		\end{equation}
		Now (\ref{L1l1e2}) leads to $\dot{Z}(s)>0$ by (\ref{dZs}).
	\end{proof}
	
By (\ref{chars}) and Lemma \ref{L1l1}, under (\ref{Lcon2}), we have that $Z(s)\in[0,1]$ for any $s\in[0,t]$ if $z\in[0,1]$. Note that $\tF{\rho}(t,z)$ is monotone increasing with respect to $z$. We have
	\begin{equation}
		\label{L1pe1}
		\begin{aligned}
			\frac{\d}{\d s}\tG{g(s)}^{\frac{1}{d+1}}(Z(s))=&\frac{1}{d+1}\frac{\partial_t \tG{g(s)}(Z(s))+\dot{Z}(s)\partial_z\tG{g(s)}(Z(s))}{\tG{g(s)}^{\frac{d}{d+1}}(Z(s))}\\
			\leqslant & C\tF{\rho}(s,Z(s))\\
			\leqslant & C\sup_{0\leqslant y\leqslant 1}\tF{\rho}(s,y)\\
			\leqslant & C\tF{\rho}(s,1)\\
			=& CF[\rho](s,\lambda(s)).
		\end{aligned}
	\end{equation}
	Taking integration over (\ref{L1pe1}),  it holds by definition of $\widetilde{G},\widetilde{F}$ in (\ref{tGFdef}) that
	\begin{equation}
		\label{L1pe2}
		\begin{aligned}
			\G{g(s)}^{\frac{1}{d+1}}=&\tG{g(s)}^{\frac{1}{d+1}}(Z(s))\\
			\leqslant& \tG{g(0)}^{\frac{1}{d+1}}(Z(0))+C\int_0^sF[\rho](\tau,\lambda(\tau))\d \tau\\
			\leqslant &\G{g(0)}^{\frac{1}{d+1}}(\lambda_1)+C\int_0^s4C_1\jbr{\tau}^{-\sigma+1}\d \tau\\
			\leqslant &\varepsilon_0^{\frac{1}{d+1}}+4CC_1\varepsilon_0^{\frac{1}{d+1}}\int_0^s\jbr{\tau}^{-\sigma+1}\d \tau\\
			\leqslant &(1+4CC_1)\varepsilon_0^{\frac{1}{d+1}}.
		\end{aligned}
	\end{equation}
	Let's name the constant $C$ in the last row of (\ref{L1e2}) as $\bar{C}_2$. Suppose
	\begin{equation}
		\label{con3}
		\tag{Assumption 3}
		1+\bar{C}_2C_1\leqslant (4C_2)^{\frac{1}{d+1}}.
	\end{equation}
	Let $s=t$ and $z=1$ in (\ref{L1pe2}), then we have
	\begin{equation*}
		\G{g(t)}(\lambda(t))\leqslant(1+\bar{C}_2C_1)^{d+1}\varepsilon_0\leqslant 4C_2\varepsilon_0.
	\end{equation*}
	Finally, since $t\in[0,T]$ is arbitrary, (\ref{L1e2}) holds under the assumptions (\ref{Lcon2}) and (\ref{con3}).
	
	In conclusion, Lemma \ref{L1} holds if the constants $C_1$, $C_2$, and $\varepsilon_0$ satisfy both (\ref{Lcon2}) and (\ref{con3}).
	
	\subsection{Proof of Lemma \ref{L2}}
	Taking integration from $0$ to $t$ over (\ref{L41e1}), (\ref{L41e2}) with respect to the time variable and then letting $\eta=kt$, we have
	\begin{equation}
		\label{L2pe1}
		\begin{aligned}
			&\hrho_+(t,k)+\varepsilon\int_0^t\hrho(s,k)(t-s)\hmu\left(k(t-s)\right)\d s \\
			=&\hf_+^0(k,kt)-\frac{\varepsilon}{(2\pi)^d}\sum_{l\in \Z^d\setminus  \{0\}}\int_0^tk(t-s)\hE(s,l)\hg_+(s,k-l,kt-ls)\d s\\
			=&\hS_+(t,k),
		\end{aligned}
	\end{equation}
	and
	\begin{equation}
		\label{L2pe2}
		\begin{aligned}
			&\hrho_-(t,k)-\int_0^t\hrho(s,k)(t-s)\hmu\left(k(t-s)\right)\d s \\
			=&\hf_-^0(k,kt)+\frac{1}{(2\pi)^d}\sum_{l\in \Z^d\setminus  \{0\}}\int_0^tk(t-s)\hE(s,l)\hg_-(s,k-l,kt-ls)\d s\\
			=&\hS_-(t,k).
		\end{aligned}
	\end{equation}
	Also, as for the analysis in linear case, let $\hS(t,k)=\hS_+(t,k)-\hS_-(t,k)$. By Theorem \ref{linearthm} and remark \ref{R1}, we have that for $t\geqslant 0$,
	\begin{equation}
		\label{L2pe3}
		F[\rho](t,\lambda(t))\leqslant F[S](t,\lambda(t))+C\int_{0}^{t}e^{-\frac{1}{4}\theta_1(t-s)} F[S](s,\lambda(s))\d s.
	\end{equation}
	Here we use that $\lambda(t)$ is a monotone decreasing function on $t$.
	\begin{lemma}
		\label{L2l1}
		Under the assumptions of Lemma \ref{L2},
		\begin{equation}
			\label{L2pe4}
			F[S](t,\lambda(t))\leqslant e^{-\lambda_1\frac{\jbr{t}}{2}}\G{f^0}^{\frac{1}{d+1}}(\lambda_1)+C (C_2\varepsilon_0)^{\frac{1}{d+1}}\jbr{t}^{-\sigma+1}\sup_{0\leqslant s\leqslant t}\left\{F[\rho](s,\lambda(s))\jbr{s}^{\sigma-1}\right\}.
		\end{equation}
	\end{lemma}
	\begin{proof}
		By (\ref{L2pe1}) and (\ref{L2pe2}) and $\hS(t,k)=\hS_+(t,k)=\hS_-(t,k)$, we have
		\begin{equation*}
			\begin{aligned}
				\hS(t,k)=&\hf_+^0(k,kt)-\hf_-^0(k,kt)\\
				&-\frac{\varepsilon}{(2\pi)^d}\sum_{l\in \Z^d\setminus  \{0\}}\int_{0}^{t}k(t-s)\hE(s,l)\hg_+(s,k-l,kt-ls)\d s\\
				&-\frac{1}{(2\pi)^d}\sum_{l\in \Z^d\setminus  \{0\}}\int_0^tk(t-s)\hE(s,l)\hg_-(s,k-l,kt-ls)\d s.
			\end{aligned}
		\end{equation*}
		To estimate $F[S](t,\lambda(t))$, we can write
		\begin{equation}
			\label{L2pe5}
			\begin{aligned}
				&e^{\lambda(t)\jbr{k,kt}}\abs{\hS(t,k)}\jbr{k,kt}^\sigma\abs{k}^{-\alpha}\\
				\leqslant &e^{\lambda(t)\jbr{k,kt}}\abs{\hf_+^0(k,kt)-\hf_-^0(k,kt)} \jbr{k,kt}^\sigma\abs{k}^{-\alpha}\\
				&+\frac{1}{(2\pi)^d}\sum_{l\in \Z^d\setminus \{0\}}\int_{0}^te^{\lambda(t)\jbr{k,kt}} \jbr{k,kt}^\sigma\abs{k}^{-\alpha}\frac{\abs{k}(t-s)}{\abs{l}}\hrho(s,l) \\
				&\abs{\varepsilon\hg_+(s,k-l,kt-ls)+\hg_-(s,k-l,kt-ls)}\d s\\
				=& I(t,k)+\frac{1}{(2\pi)^d} R(t,k).
			\end{aligned}
		\end{equation}
		Therefore it holds
		\begin{equation}
			\label{L2pe6}
			F[S](t,\lambda(t))\leqslant \sup_{k\in \Z^d\setminus\{0\}}I(t,k)+\frac{1}{(2\pi)^d} \sup_{k\in \Z^d\setminus\{0\}}R(t,k).
		\end{equation}
		Firstly, let's estimate $I(t,k)$. Since $\lambda(t)\leqslant 2\lambda_0\leqslant\frac{\lambda_1}{2}$, by (\ref{L2pe5}), it follows that
		\begin{equation}
			\label{L2Iest}
			\begin{aligned}
				\sup_{k\in \Z^d\setminus\{0\}}I(t,k)\leqslant &\sup_{k\in \Z^d\setminus\{0\}} e^{-(\lambda_1-\lambda(t))\jbr{k,kt}}e^{\lambda_1\jbr{k,kt}}\abs{\hf_+^0(k,kt)-\hf_-^0(k,kt)} \jbr{k,kt}^\sigma\abs{k}^{-\alpha}\\
				\leqslant &e^{-\frac{\lambda_1}{2}\jbr{t}}\sup_{k\in \Z^d\setminus\{0\}}\left\{\sup_\eta e^{\lambda_1\jbr{k,\eta}}\abs{\hf_+^0(k,\eta)-\hf_-^0(k,\eta)} \jbr{k,\eta}^\sigma\right\}\\
				\leqslant &Ce^{-\frac{\lambda_1}{2}\jbr{t}}\G{f^0}^{\frac{1}{d+1}} (\lambda_1).
			\end{aligned}
		\end{equation}
		Here, in the last inequality we have used Corollary \ref{re24}.
		
		Secondly, we are going to estimate $R(t,k)$. Using
		\begin{equation*}
			\jbr{k,kt}\leqslant \jbr{l,ls}+\jbr{k-l,kt-ls},
		\end{equation*}
		we have
		\begin{equation}
			\label{L2pe7}
			e^{\lambda(t)\jbr{k,kt}}\leqslant e^{\left(\lambda(t)-\lambda(s)\right)\jbr{k,kt}}e^{\lambda(s)\jbr{l,ls}}e^{\lambda(s)\jbr{k-l,kt-ls}}.
		\end{equation}
		So combining (\ref{L2pe5}) and (\ref{L2pe7}), we have
		\begin{equation}
			\label{L2pe8}
			\begin{aligned}
				R(t,k) &\leqslant \sum_{l\in \Z^d\setminus  \{0\}}\int_0^t \abs{k}^{1-\alpha}(t-s)e^{\left(\lambda(t)-\lambda(s)\right)\jbr{k,kt}}\jbr{k,kt}^\sigma \jbr{l,ls}^{-\sigma}\jbr{k-l,kt-ls}^{-\sigma} \abs{l}^{-1+\alpha}\\
				& \times e^{\lambda(s)\jbr{l,ls}}\abs{\hrho(s,l)}\jbr{l,ls}^\sigma \abs{l}^{-\alpha}\\ 
				&\times e^{\lambda(s)\jbr{k-l,kt-ls}}\left[\abs{\hg_+(s,k-l,kt-ls)}+\abs{\hg_-(s,k-l,kt-ls)}\right] \jbr{k-l,kt-ls}^\sigma \d s.
			\end{aligned}
		\end{equation}
		By Corollary \ref{re23} and (\ref{L2e1}), we have
		\begin{equation}
			\label{L2R1}
			\begin{aligned}
				&\sup_\eta e^{\lambda(s)\jbr{k-l,\eta}}\left[\abs{\hg_+(s,k-l,\eta)}+\abs{\hg_-(s,k-l,\eta)}\right] \jbr{k-l,\eta}^\sigma \\
				\leqslant &C \G{g(t)}^{\frac{1}{d+1}}(\lambda(s))\\
				\leqslant &C (4C_2\varepsilon_0)^{\frac{1}{d+1}}.
			\end{aligned}
		\end{equation}
		So the third term of (\ref{L2pe8}) can be bounded using  (\ref{L2R1}) as 
		\begin{equation}
			\label{L2R1est}
			\begin{aligned}
				& e^{\lambda(s)\jbr{k-l,kt-ls}}\left[\abs{\hg_+(s,k-l,kt-ls)}+\abs{\hg_-(s,k-l,kt-ls)}\right] \jbr{k-l,kt-ls}^\sigma\\
				\leqslant &C (4C_2\varepsilon_0)^{\frac{1}{d+1}}.
			\end{aligned}
		\end{equation}
		Then we define
		\begin{equation*}
			C_{k,l}(t,s)=\abs{k}^{1-\alpha}(t-s)e^{\left(\lambda(t)-\lambda(s)\right)\jbr{k,kt}^\gamma}\jbr{k,kt}^\sigma \jbr{l,ls}^{-\sigma}\jbr{k-l,kt-ls}^{-\sigma} \abs{l}^{-1+\alpha}.
		\end{equation*}
		Here, $\gamma$ is the Gevrey index as we mentioned in (\ref{Fdef}) and (\ref{Gdef}). Provided $3\gamma>1+2\delta$, we can use \cite[eqn. 4.23]{GNR} to deduce
		\begin{equation}
			\label{GNR423}
			\sup_{k\in \Z^d\setminus\{0\}}\sum_{l\in \Z^d\setminus \{0\}}\int_0^tC_{k,l}(t,s)\jbr{s}^{-\sigma+1}\d s\leqslant C\jbr{t}^{-\sigma+1}.
		\end{equation}
		For the non-linear damping, we only deal with the case that the Gevrey index is equal to $1$, which means $\gamma=1$, so (\ref{GNR423}) still holds since $\delta<1$.
		By (\ref{GNR423}), together with (\ref{L2R1est}) and (\ref{L2pe8}), we have
		\begin{equation}
			\label{L2Rest}
			\begin{aligned}
				& \sup_{k\in \Z^d\setminus\{0\}}R(t,k)\\
				\leqslant & C (4C_2\varepsilon_0)^{\frac{1}{d+1}}\left(\sup_{0\leqslant s\leqslant t}F[\rho](s,\lambda(s))\jbr{s}^{\sigma-1}\right) \sup_{k\in \Z^d\setminus\{0\}}\sum_{l\in \Z^d\setminus \{0\}}\int_0^tC_{k,l}(t,s)\jbr{s}^{-\sigma+1}\d s\\
				\leqslant &C (4C_2\varepsilon_0)^{\frac{1}{d+1}}\jbr{t}^{-\sigma+1}\left(\sup_{0\leqslant s\leqslant t}F[\rho](s,\lambda(s))\jbr{s}^{\sigma-1}\right).
			\end{aligned}
		\end{equation}
		Combing (\ref{L2pe6}), (\ref{L2Iest}) and (\ref{L2Rest}), we complete the proof of (\ref{L2pe4}).
	\end{proof}
	
	Now, we would use Lemma \ref{L2l1} to prove Lemma \ref{L2}. For convenience, define
	\begin{equation}
		\label{zetadef}
		\zeta(t)=\sup_{0\leqslant s\leqslant t}F[\rho](s,\lambda(s))\jbr{s}^{\sigma-1}.
	\end{equation}
	Clearly, $\zeta(t)$ is a monotone increasing function. So, we can write (\ref{L2pe4}) into
	\begin{equation}
		\label{L2pe9}
		F[S](t,\lambda(t))\leqslant e^{-\lambda_1\frac{\jbr{t}}{2}}\G{f^0}^{\frac{1}{d+1}}(\lambda_1)+C (C_2\varepsilon_0)^{\frac{1}{d+1}}\jbr{t}^{-\sigma+1}\zeta(t).
	\end{equation}
	Combing (\ref{L2pe3}) and (\ref{L2pe9}), we have
	\begin{equation}
		\label{L2pe10}
		\begin{aligned}
			F[\rho](t,\lambda(t))\leqslant & e^{-\lambda_1\frac{\jbr{t}}{2}}\G{f^0}^{\frac{1}{d+1}}(\lambda_1)+C (C_2\varepsilon_0)^{\frac{1}{d+1}}\jbr{t}^{-\sigma+1}\zeta(t) \\
			&+C\int_0^t e^{-\frac{1}{4}\theta_1(t-s)}e^{-\lambda_1\frac{\jbr{s}}{2}}\d s \G{f^0}^{\frac{1}{d+1}}(\lambda_1)\\
			&+C (C_2\varepsilon_0)^{\frac{1}{d+1}}\zeta(t)\int_0^t e^{-\frac{1}{4}\theta_1(t-s)}\jbr{s}^{-\sigma+1}\d s.
		\end{aligned}
	\end{equation}
	By \cite[Lemmas 4.8 and 4.9]{GI}, it holds that
	\begin{equation}
		\label{GI48}
		\int_0^t e^{-\frac{1}{4}\theta_1(t-s)}e^{-\lambda_1\frac{\jbr{s}}{2}}\d s \leqslant Ce^{-\frac{\lambda_1}{2}\jbr{t}},
	\end{equation}
	and
	\begin{equation}
		\label{GI49}
		\int_0^t e^{-\frac{1}{4}\theta_1(t-s)}\jbr{s}^{-\sigma+1}\d s\leqslant C\jbr{t}^{-\sigma+1}.
	\end{equation}
	Now let's use (\ref{GI48}) and (\ref{GI49}) to continue the estimate in (\ref{L2pe10}), and also note (\ref{Ginitial}). It then holds that
	\begin{equation}
		\label{L2Frho}
		F[\rho](t,\lambda(t))\leqslant Ce^{-\lambda_1\frac{\jbr{s}}{2}}\varepsilon_0^{\frac{1}{d+1}} +C (C_2\varepsilon_0)^{\frac{1}{d+1}}\jbr{t}^{-\sigma+1}\zeta(t).
	\end{equation}
	Now we recall the definition of $\zeta(t)$ (\ref{zetadef}), and use (\ref{L2Frho}), to get
	\begin{equation}
		\label{zetaest}
		\begin{aligned}
			\zeta(t)=&\sup_{0\leqslant s\leqslant t}F[\rho](s,\lambda(s))\jbr{s}^{\sigma-1}\\
			\leqslant &C\varepsilon_0^{\frac{1}{d+1}}\sup_{0\leqslant s\leqslant t} \left\{e^{-\lambda_1\frac{\jbr{s}}{2}}\jbr{s}^{\sigma-1}\right\}+C (C_2\varepsilon_0)^{\frac{1}{d+1}}\zeta(t)\\
			\leqslant &C\varepsilon_0^{\frac{1}{d+1}}+CC_2^{\frac{1}{d+1}}\varepsilon_0^{\frac{1}{d+1}} \zeta(t).
		\end{aligned}
	\end{equation}
	Let's name the constant $C$ in the last row of (\ref{zetaest}) as $\bar{C}_3$, then we have
	\begin{equation}
		\label{L2zetaest}
		\zeta(t)\leqslant \frac{\bar{C}_3}{1-\bar{C}_3C_2^{\frac{1}{d+1}}\varepsilon_0^{\frac{1}{d+1}}}\varepsilon_0^{\frac{1}{d+1}}.
	\end{equation}
	Assume that
	\begin{equation}
		\tag{Assumption 4}
		\label{con4}
		\frac{\bar{C}_3}{1-\bar{C}_3C_2^{\frac{1}{d+1}}\varepsilon_0^{\frac{1}{d+1}}}\leqslant 2C_1.
	\end{equation}
	Combining (\ref{L2zetaest}) and (\ref{zetadef}), we can deduce (\ref{L2e2}).
	
	In conclusion, Lemma \ref{L2} holds provided that (\ref{con4}) holds.
	
	\subsection{Proof of the main theorem}
	To prove Theorem \ref{mainthm}, we need fix $C_1$, $C_2$ and $\varepsilon_0$ to make sure Lemmas \ref{L1} and \ref{L2} as well as  Proposition \ref{p4} hold. By the analysis in Sections 4.2 to 4.4, Lemmas \ref{L1}, \ref{L2} and Proposition \ref{p4} hold if $C_1$, $C_2$ and $\varepsilon_0$ satisfy all the assumptions (\ref{con1}), (\ref{Lcon2}), (\ref{con3}) and (\ref{con4}), i.e., 
	\begin{equation}
		\tag{\ref{con1}}
		C_0\leqslant 2C_1,
	\end{equation}
	\begin{equation}
		\tag{\ref{Lcon2}}
		4\bar{C}_1C_1\varepsilon_0^{\frac{1}{d+1}}<\delta\lambda_0,
	\end{equation}
	\begin{equation}
		\tag{\ref{con3}}
		1+\bar{C}_2C_1\leqslant (4C_2)^{\frac{1}{d+1}},
	\end{equation}
	\begin{equation}
		\tag{\ref{con4}}
		\frac{\bar{C}_3}{1-\bar{C}_3C_2^{\frac{1}{d+1}}\varepsilon_0^{\frac{1}{d+1}}}\leqslant 2C_1.
	\end{equation}
	To fix these three constants, we firstly find a $C_1$ such that
	\begin{equation}
		\label{M1}
		C_1\geqslant \bar{C}_3
	\end{equation}
	and
	\begin{equation}
		\label{M2}
		C_1\geqslant \frac{C_0}{2}.
	\end{equation}
	Secondly, find a $C_2$ such that
	\begin{equation}
		\label{M3}
		C_2\geqslant \frac{(1+\bar{C}_2C_1)^{d+1}}{4}.
	\end{equation}
	Thirdly, find $\varepsilon_0$ such that
	\begin{equation}
		\label{M4}
		\varepsilon_0^{\frac{1}{d+1}}<\frac{\delta \lambda_0}{4\bar{C}_1C_1}
	\end{equation}
	and
	\begin{equation}
		\label{M5}
		\varepsilon_0^{\frac{1}{d+1}}\leqslant \frac{1}{2\bar{C}_3C_2^{\frac{1}{d+1}}}.
	\end{equation}
It is obvious to see such choice is workable. Then we will show that such $C_1$, $C_2$ and $\varepsilon_0$ satisfy the four assumptions. In fact, by (\ref{M2}), (\ref{con1}) holds; by (\ref{M4}), there holds (\ref{Lcon2}); by (\ref{M3}), we can deduce (\ref{con3}); by (\ref{M5}), we have that
	\begin{equation}
		\label{M6}
		\bar{C}_3C_2^{\frac{1}{d+1}}\varepsilon_0^{\frac{1}{d+1}}\leqslant \frac{1}{2}.
	\end{equation}
	Then, combining (\ref{M1}) and (\ref{M6}), we  confirm that (\ref{con4}) holds.
	
	Now we have that Proposition \ref{p4} holds. By (\ref{FG}), we have
	\begin{equation*}
		\label{Fresult}
		F[\rho](t,\lambda(t))\leqslant C_0\G{g(t)}^{\frac{1}{d+1}}(\lambda(t))\leqslant C\varepsilon_0^{\frac{1}{d+1}},
	\end{equation*}
	which means for any $k\neq 0$, 
	\begin{equation}
		\label{result2}
		e^{\lambda(t)\jbr{k,kt}}\abs{\hrho(t,k)}\jbr{k,kt}^\sigma\abs{k}^{-\alpha}\leqslant C\varepsilon_0 ^{\frac{1}{d+1}}.
	\end{equation}
	Since $\sigma>d+1$ and $\alpha<\frac{1}{d+1}$, (\ref{result2}) becomes
	\begin{equation}
		\label{result3}
		\begin{aligned}
			\abs{\hrho(t,k)}\leqslant & C e^{\lambda(t)\jbr{k,kt}}\varepsilon_0^{\frac{1}{d+1}}\\
			\leqslant & Ce^{-\lambda_0\jbr{t}}\varepsilon_0^{\frac{1}{d+1}}e^{-\lambda_0\abs{k}}.
		\end{aligned}
	\end{equation}
	Finally, by $ik\cdot \hE(t,k)=\hrho(t,k)$, togethor with (\ref{result3}), we have
	\begin{equation}
		\label{result}
		\begin{aligned}
			\abs{E(t,x)}\leqslant& \frac{1}{(2\pi)^d}\sum_{k\in\Z^d\setminus \{0\}} \abs{\hE(t,k)}\\
			\leqslant &\frac{1}{(2\pi)^d}\sum_{k\in\Z^d\setminus \{0\}}\frac{1}{\abs{k}} \abs{\hrho(t,k)}\\
			\leqslant &Ce^{-\lambda_0\jbr{t}}\varepsilon_0^{\frac{1}{d+1}}\sum_{k\in\Z^d\setminus \{0\}} \frac{1}{\abs{k}}e^{-\lambda_0\abs{k}}\\
			\leqslant &Ce^{-\lambda_0\jbr{t}}\varepsilon_0^{\frac{1}{d+1}}\sum_{k\in\Z^d\setminus \{0\}} e^{-\frac{\lambda_0}{2}\abs{k}}\\
			\leqslant & Ce^{-\lambda_0\jbr{t}}\varepsilon_0^{\frac{1}{d+1}}.
		\end{aligned}
	\end{equation}
Then, the estimate (\ref{result}) above indicates the Landau damping of non-linear Vlasov-Poisson system. This completes the proof of Theorem \ref{mainthm}.
	
\medskip
\noindent{\bf Acknowledgements:}  The research of Renjun Duan was partially supported by the General Research Fund (Project No.~14303523) from RGC of Hong Kong.

\medskip

\noindent{\bf Data Availability:} There are no data associated to this work to be made available

\medskip

\noindent{\bf Conflict of Interest:} The authors declare that they have no conflict of interest.

\end{document}